\documentclass[11pt]{article}
\usepackage {amssymb} 
\usepackage {amsmath}
\usepackage{amsthm}
\usepackage{graphicx}
\usepackage {amscd}
\usepackage{subfigure}
\usepackage{tikz-cd}

\usepackage{color}
\usepackage{multirow}
\usepackage{tabu}

\usepackage{diagbox}
\usepackage[colorlinks, citecolor=blue]{hyperref}

\usepackage{geometry}  

\newcommand{\vol}{{\rm vol}}
\newcommand{\ord}{{\rm ord}}
\newcommand{\fm}{\mathfrak{m}}
\newcommand{\fa}{\mathfrak{a}}
\newcommand{\cO}{\mathcal{O}}
\newcommand{\bR}{\mathbb{R}}
\newcommand{\bC}{\mathbb{C}}
\newcommand{\bZ}{\mathbb{Z}}
\newcommand{\lct}{{\rm lct}}

\newcommand{\Val}{{\rm Val}}

\newcommand{\hvol}{{\widehat{\rm vol}}}

\newcommand{\cF}{{\mathcal{F}}}

\newcommand{\cI}{{\mathcal{I}}}

\newcommand{\cJ}{{\mathcal{J}}}

\newcommand{\bQ}{{\mathbb{Q}}}

\newcommand{\cX}{{\mathcal{X}}}
\newcommand{\cV}{{\mathcal{V}}}
\newcommand{\cL}{{\mathcal{L}}}
\newcommand{\ft}{{\mathfrak{t}}}

\newcommand{\bP}{{\mathbb{P}}}

\newcommand{\cC}{{\mathcal{C}}}
\newcommand{\cS}{{\mathcal{S}}}

\newcommand{\bld}{{\bf ld}}

\newcommand{\sddb}{{\sqrt{-1}\partial\bar{\partial}}}
\newcommand{\cD}{{\mathcal{D}}}
\newcommand{\vphi}{{\varphi}}
\newcommand{\tS}{{\widetilde{S}}}

\newcommand{\cH}{{\mathcal{H}}}
\newcommand{\ocC}{{\overline{\mathcal{C}}}}
\newcommand{\orb}{{\rm orb}}

\newcommand{\wtd}{\widetilde}
\newcommand{\cE}{{\mathcal{E}}}
\newcommand{\cP}{{\mathcal{P}}}
\newcommand{\Spec}{\mathrm{Spec}}
\newcommand{\Sym}{\mathrm{Sym}}
\newcommand{\tr}{\mathrm{tr}}

\newcommand\lDing{\operatorname{log-Ding}}
\newcommand{\lCM}{\operatorname{log-CM}}

\newtheorem{thm}{Theorem}[section]

\newtheorem{lem}[thm]{Lemma}
\newtheorem{cor}[thm]{Corollary}
\newtheorem{defn}[thm]{Definition}

\newtheorem{prop}[thm]{Proposition}

\newtheorem{rem}[thm]{Remark}
\newtheorem{exmp}[thm]{Example}

\begin{document}

\title{K\"{a}hler-Einstein metrics and volume minimization}
\author{Chi Li, Yuchen Liu}

\maketitle{}

\abstract{
We prove that if a $\bQ$-Fano variety $V$ specially degenerates to a 
K\"{a}hler-Einstein $\bQ$-Fano variety $V$, then for any ample Cartier 
divisor $H=-r^{-1} K_V$ with $r\in \bQ_{>0}$, the normalized volume 
$\hvol(v)=A_{\cC}^n(v)\cdot \vol(v)$ is globally minimized at the canonical 
valuation $\ord_V$ among all real valuations which are centered at the vertex 
of the affine cone $\cC:=C(V,H)$.
This is also generalized to the logarithmic and the orbifold setting. As a consequence, we complete the confirmation of a conjecture in \cite{Li15a} on an equivalent characterization of K-semistability for any smooth Fano manifold. We also prove that the valuation associated to the Reeb vector field of a smooth Sasaki-Einstein metric minimizes $\hvol$ over the corresponding K\"{a}hler cone. These results strengthen the minimization result of Martelli-Sparks-Yau \cite{MSY08}. }

\tableofcontents

\section{Introduction}

\subsection{Motivation and background}
The study of K\"{a}hler-Einstein (KE) metrics is very active recently. In particular, the Yau-Tian-Donaldson correspondence has been established for smooth Fano manifolds (see \cite{Tia97}, \cite{Bm12}, \cite{CDS15, Tia15}). This correspondence says that a smooth Fano manifold 
admits a K\"{a}hler-Einstein metric if and only if it's K-polystable. The existence of K\"{a}hler-Einstein metrics of positive curvature is a global question. However, it can be related to a local question by considering the affine cone over the underlying Fano variety with the polarization given by a positive Cartier multiple of the anti canonical divisor. The affine cone is singular except for the case that the Fano variety is $\bP^n$ and the polarization one uses is $\cO_{\bP^n}(1)=\frac{1}{n+1}(-K_{\bP^n})$. Such cone 
singularities are basic examples of $\bQ$-Gorenstein klt singularities (see \cite{Kol13}). It's natural to ask what information does the K\"{a}hler-Einstein condition provide for the associated cone singularities. An answer is that there are K\"{a}hler Ricci-flat cone metrics on these affine cones over KE Fano manifolds. These Ricci-flat cone metrics are rotationally symmetric and are easily obtained by solving an ODE with respect to radius functions determined by the K\"{a}hler-Einstein metrics. 
In this way, the study of K\"{a}hler-Einstein metrics can be put into a broader setting of K\"{a}hler Ricci-flat cone metrics. The latter has also been studied extensively recently. Given an affine cone singularity, $(\cC, o)$ any K\"{a}hler cone metric gives rise to a radius function $r$ and its associated Reeb vector field $\partial_\theta:=J(r\partial_r)$ where $J$ is the complex structure on the regular part of the affine cone. If $\partial_\theta$ has closed orbits, then $r\partial_r$ and $\partial_\theta$ generates an effective $\bC^*$-action.
This is called the quasi-regular case which is relevant to our discussion in this paper. The Ricci-flat condition is then translated to the condition that the quotient $(\cC\setminus \{o\})/\bC^*$ is a K\"{a}hler-Einstein Fano orbifold $(V, \Delta)$, at least when $\cC$ has isolated singularity at $o$.

Whether there exists such a good radius function or an associated good $\bC^*$-action is a delicate question. 
In the case when there is an effective torus $T\cong (\bC^*)^d$ action on $\cC$,  it was proved in \cite{GMSY07, MSY08} that the $ \bC^*$-action corresponding to the quasi-regular K\"{a}hler Ricci-flat cone metric should minimize a normalized volume functional. The normalized volume functional is defined on the space of Reeb vector fields which is a conic subset of the Lie algebra of $T$. In \cite{Li15a}, this normalized volume functional was re-interpreted as the normalized volume of the valuations associated to the $\bC^*$-actions. This allows us to define the normalized volume functional, denoted by $\hvol(v)$, for any valuation $v$. More precisely for any valuation $v$ we define $\hvol(v)=A_\cC(v)^n\cdot \vol(v)$. $A_\cC(v)$ is the log discrepancy of $v$. $\vol(v)$ is the volume of $v$, which is finite if $v$ is centered at a closed point. Using this new point of view, one can consider the normalized volume as a function on the space of valuations which are centered at $o\in \cC$. 
Moreover, we can consider any $\bQ$-Gorenstein klt singularity, which is not necessarily an affine cone singularity or equipped with any $\bC^*$-action. We then ask whether there is a minimizer of this normalized volume functional, and, if there is a minimizer, what is its geometric meaning (see \cite{Li15a}). Although these are purely algebro-geometric questions, we speculated in \cite{Li15a} that answering them could help to understand  {\it metric tangent cones} of K\"{a}hler-Einstein varieties (cf. \cite{DS15}).
 
As a continuation of \cite{Li15a, Li15b}, we study these two questions for affine cones over ``K-semistable" $\bQ$-Fano  varieties. 
In this paper we will show that for a $\bQ$-Fano variety $V$, if $V$ (specially) degenerates to a K\"{a}hler-Einstein Fano variety, then the canonical valuation $\ord_V$ is actually a global minimizer of $\hvol$. As a consequence, we obtain an equivalent characterization of K-semistability for smooth Fano manifold, thus completing the confirmation of a conjecture in \cite{Li15a}. 

On the other hand, Martelli-Sparks-Yau's minimization result in \cite{MSY08} works for a general (smooth) Sasaki-Einstein manifold whose associated Reeb vector field does not have to generate a $\bC^*$-action. If the Reeb vector field generates a torus action of rank bigger than $1$, then the Sasaki-Einstein metric is called irregular. By using approximation arguments, we can indeed extend our result to prove a minimization result in the irregular case that generalizes Martelli-Sparks-Yau's result.  


\subsection{Statement of main results}
Let $(V^{n-1}, E^{n-2})$ be a log-Fano pair. By this we mean that
$-(K_V+E)$ is an ample $\bQ$-Cartier divisor and $(V, E)$ has klt 
singularities with $E$ effective. 
Assume $H=-r^{-1}(K_V+E)$ is an ample Cartier divisor for an 
$r\in \bQ_{>0}$. We will denote the affine cone by 
$\cC:=C(V, H)={\rm Spec}\bigoplus_{k=0}^{+\infty}H^0(V, kH)$.
Notice that $\cC$ has $\bQ$-Gorenstein klt singularities at the vertex
$o$ (see \cite[Lemma 3.1]{Kol13}). $(\cC, o)$ has a natural 
$\bC^*$-action. Denote by $v_0$ the canonical $\bC^*$-invariant 
divisorial valuation $\ord_V$ where $V$ is considered as the 
exceptional divisor of the blow up $Bl_o \cC\rightarrow \cC$. Denote by $\cE$ the effective divisor on $\cC$ corresponding to $E$, which is the closure of the divisor $f^{-1}E$ under the projection $f: \cC-\{o\}\rightarrow V$.
Using these notations, the following is the first main result of this paper.  

\begin{thm}\label{thmlog}
With the above notation, 
if $(V, E)$ is a conical K\"{a}hler-Einstein log-Fano pair, then $\hvol_{(\cC,\cE)}$ is globally minimized over $(\cC, o)$ at $\ord_V$.
\end{thm}
For the definition of conical K\"{a}hler-Einstein potentials/metrics and the normalized volume $\hvol_{(\cC, \cE)}$ see Section \ref{secpre}. 
Letting $E=0$, we get a non-log version:
\begin{cor}\label{corKEcone}
If $V$ is K\"{a}hler-Einstein $\bQ$-Fano variety, then $\hvol$ is globally minimized over $(\cC, o)$ at $\ord_V$.
\end{cor}


There are several steps to prove Theorem \ref{thmlog}. We first prove
that there is a conical K\"{a}hler-Einstein metric on the projective
cone $(\ocC, \cE+(1-\beta)V_\infty)$ with $\beta=r/n$. See section \ref{secprojcone} for the relevant notations. (Notice that
$0<\beta\leq 1$ by Lemma \ref{fanoindex}.) This is a generalization
of the construction in \cite[Lemma 3]{Li13}. Next we use Berman's result to conclude that $(\ocC, \cE+(1-\beta)V_\infty)$ is log-K-polystable and hence log-K-semistable, or equivalently log-Ding-semistable. Then we apply the log version 
of Fujita's result to the filtration associated to any valuation. At this point, we can proceed in two ways to complete the proof. For the first (quick) proof we use similar arguments to those used in \cite{Fuj15, Liu16} to obtain an estimate of $\hvol_{(\cC, \cE)}$ which turns out be sharp. For the second proof, we interpret the expression in log-Fujita as the derivative of normalized volumes in the same spirit as in \cite{Li15a, Li15b} and use a convexity argument to conclude the proof (see Section \ref{secdervol}).

We can also deal with the semistable case. See section \ref{secKstability} for the definition of special degenerations.

\begin{thm}\label{thmsemi}
Assume $V$ is a $\bQ$-Fano variety. Assume either of the following two conditions is satisfied:
\begin{enumerate}
\item $V$ specially degenerates to a K\"{a}hler-Einstein $\bQ$-Fano variety;
\item $(\ocC, (1-\frac{r}{n})V_\infty)$ specially degenerates to a conical K\"{a}hler-Einstein pair.
\end{enumerate}
Then $\hvol$ is globally minimized at $\ord_V$ over $(\cC, o)$.
\end{thm}
\begin{rem}
Note that the second condition does not necessarily follow from the first one (see \cite[{\it Aside} in Section 3.8]{Kol13}). However the first condition should imply the second one up to a branched covering of the projective cone.
\end{rem}
When $V$ is a smooth Fano manifold, by the deep results in \cite{CDS15, Tia15} we know that $V$ being K-semistable is equivalent to the condition that $V$ degenerates to a K\"{a}hler-Einstein Fano variety. 
This combined with the above theorem and the results in \cite{Li15b} allows us to complete the confirmation of a conjecture in \cite{Li15a}.
\begin{cor}\label{corsemi}
Assume $V$ is a smooth Fano manifold. Then $V$ is K-semistable if and only if $\hvol$ is globally minimized
at $\ord_V$ over $(\cC, o)$.
 \end{cor}
 \begin{rem}
We expect Corollary \ref{corsemi} to be true for any $\bQ$-Fano variety. In \cite{Li15b}, the first author proved that one direction is true for any $\bQ$-Fano variety. Actually the following stronger criterion for K-semistability was proved there: if $\hvol$ is minimized at $\ord_V$ among $\bC^*$-equivariant divisorial valuations, then $V$ is K-semistable. 
\end{rem}
The same argument in the proof of Theorem \ref{thmlog} also allows us to get orbifold versions of the above results. In particular, we have
the following combined version. See Section \ref{AppSeif} for the notations used here. 
\begin{thm}\label{thmorb}
Let $(V, \Delta=\sum_{i} (1-\frac{1}{m_i})D_i)$ be a smooth orbifold and $f: Y^n\rightarrow (V^{n-1}, \Delta)$ be a Seifert $\bC^*$-bundle. Assume that
the orbifold canonical class $K_{\orb}:=K_V+\Delta$ is anti ample and $c_1(Y/V)=-r^{-1} (K_V+\Delta)$ with $0<r\le n$. Denote by $\ocC_{\orb}$ the associated orbifold projective cone and by $V_\infty$ the compactifying divisor at infinity. Then the following holds:
\begin{enumerate}
\item If $(V, \Delta)$ admits an orbifold K\"{a}hler-Einstein metric of positive Ricci curvature, then there is a conical K\"{a}hler-Einstein metric on $(\ocC_\orb, (1-\frac{r}{n})V_\infty)$.
 
\item If $(\ocC_\orb, (1-\frac{r}{n})V_\infty)$ specially degenerates to a conical K\"{a}hler-Einstein pair, then over the orbifold cone $(\cC_\orb, o)$, $\hvol$ is globally minimized at $\ord_V$.
In particular, in the situation of item 1, we can choose the special degeneration as the trivial one and hence
the same conclusion holds.
\end{enumerate} 
\end{thm}
\begin{rem}
With our assumptions, the Seifert bundle $Y$ in general has quotient singularities.
By allowing $(V, \Delta)$ to be a general algebraic stack, we could weaken the condition that $Y$ has quotient singularities and get more general version of the above results. However since we don't use them in this paper, we will not discuss it here. 
\end{rem}

Finally, we can use approximation argument to cover the case of smooth Sasaki-Einstein metrics. Indeed, by using approximation method, we can prove 
\begin{thm}[=Theorem \ref{thm-irSE}]\label{thm-SE}
Let $M$ be a smooth Sasaki-Einstein manifold with the Reeb vector field $\xi$. Let $X=C(M)$ be the K\"{a}hler cone over $M$. Then the valuation associated to $\xi$ minimizes $\hvol_X$ over $\Val_{X,o}$.
\end{thm}

\begin{rem}
Theorem \ref{thmorb} and \ref{thm-SE} strength the volume minimization result in \cite{MSY08}.
\end{rem}



We emphasize here that the main common feature of proofs of the above results is that we need to work in the conical/logarithmic setting.
On the one hand we work on the projective cone over the original $\bQ$-Fano variety and construct a conical K\"{a}hler-Einstein metric using a (conical) K\"{a}hler-Einstein metric on the original $\bQ$-Fano variety. On the other hand we need to apply Berman and Fujita's results about Ding-semistability in their log versions. Although this seems innocuous at first sight, it turns out to be crucial for us to get sharp estimates. 

\section{Preliminaries}\label{secpre}

\subsection{Conical K\"{a}hler-Einstein metrics on log Fano pair}\label{seconic}

We recall some notions in the logarithmic setting following \cite{BBEGZ11}. Let $(X, D)$ be a log pair satisfying: 
\begin{itemize}
\item $D$ is an effective $\bQ$-divisor and $-(K_X+D)$ is a $\bQ$-Cartier $\bQ$-divisor; 
\item $(X, D)$ has klt singularities; 
\item $-(K_X+D)$ is ample.
\end{itemize}

From now on, let $\delta>0$ be a positive rational number such that $L=-\delta^{-1}(K_X+D)$ is Cartier. Then $L$ can be considered as a holomorphic line bundle.  A locally bounded Hermitian metric on $L$ is given by a family of locally bounded positive functions $e^{-\vphi}:=\{e^{-\vphi_i}\}$ associated to an affine covering $\{U_i\}_i$ of $X$, such that they are compatible with the transition functions $\{g_{ij}\}\in H^1(\{U_i\}, \cO_X^*)$ of $L$:
\[
\frac{e^{-\vphi_j}}{e^{-\vphi_i}}=\frac{s_j}{s_i}=g_{ij}
\]
where $s_i$ is a local generator of $L$ over $U_i$. $e^{-\vphi}$ is positively curved if $\vphi_i$ is plurisubharmonic on each affine subset $U_i$, i.e. $\vphi_i$ is lower semi-continuous and $\sddb \vphi_i\ge 0$ in the sense of currents. A locally bounded positively curved Hermitian metric $e^{-\vphi}$ on $L$ is called to be a conical K\"{a}hler-Einstein (cKE) potential 
on $(X, D; L)$ if it satisfies the following Monge-Amp\`{e}re equation:
\begin{equation}\label{eqcKE}
(\sddb \vphi)^n={\rm m}_{\delta \vphi},
\end{equation}
where we used the following notations:
\begin{itemize}
\item For the left hand side, the Monge-Amp\`{e}re measure $(\sddb\vphi)^n$ is defined in the sense of pluripotential theory;
\item To define the right hand side, we first choose any integer $k$ such that $k(K_X+D)$ is Cartier and a non vanishing local section $\sigma^k$. For example, we can choose $k=\delta^{-1}$. Then we define locally:
\[
{\rm m}_{\delta\vphi}=\left(\sigma^k \wedge \bar{\sigma}^k\right)^{1/k} \left(|\sigma^{-k}|^2 e^{-k \delta \vphi}\right)^{1/k}=:dV_{(X,D)}(\sigma) \cdot\left( |\sigma|^2 e^{-\delta \vphi}\right). 
\]
It's easy to verify that this is a globally well defined singular volume form.
If we choose a log resolution of $\pi: (Y, f_*^{-1}D) \rightarrow (X, D)$ such that $K_Y=\pi^*(K_X+D)+\sum_i a_i E_i$ with the normal crossing divisor $\bigcup_i E_i$, then locally we have:
\[
\pi^*({\rm m}_{\delta\vphi})=e^{-\psi} \prod_i |f_i|^{2 a_i} d\lambda.
\]
where $\psi$ a bounded function, $E_i=\{f_i=0\}$ and $d\lambda$ is the local Lebesgue measure on $Y$. Notice that the klt condition means that $a_i>-1$, which implies that the density function of $\pi^*({\rm m}_{\delta\vphi})$ with respect to $d\lambda$ is $L^p(d\lambda)$ for some $p>1$. We refer to \cite[Section 3]{BBEGZ11} for more details.
\end{itemize}
Correspondingly, the curvature current $\omega_\vphi:=\sddb\vphi$ is called to be a conical K\"{a}hler-Einstein metric of $(X, D)$ with Ricci curvature $\delta>0$.
If there is a cKE potential on $(X, D; L)$ then the same holds for $(X, D; \lambda L)$ for any $\lambda\in \bQ_{>0}$. So for simplicity, if the proportional constant $\delta>0$ is clear (equivalently when the cohomology class of $L$ is clear) we will just say that there is a cKE metric on $(X, D)$. 

In the case when both $K_X$ and $D$ are $\bQ$-Cartier, we can choose a local non vanishing section $e_1^{m}$ of $mK_X$ and a local non vanishing section $e_2^m$ of $mD$. We then think $e_1$ and $e_2$ as $\bQ$-sections of the $\bQ$-line bundles $K_X$ and $D$, and obtain a $\bQ$-section $\sigma=e_1\cdot e_2$ of $K_X+D$. Notice that on the  open set $U_\alpha$, we can choose $e_2^m=\frac{1}{f_\alpha}\in \cO(mD)(U_\alpha)$ where $mD=\{f_\alpha=0\}$. Then we can write:
\begin{eqnarray*}
(\sigma\wedge\sigma)^{1/m}&=&\left(e_1^m\wedge \bar{e}_1^m\right)^{1/m} \left(e_2^m\wedge \bar{e}_2^m\right)^{1/m}\\
&=&(e_1^m\wedge \bar{e}_1^m)^{1/m}\frac{1}{|f_\alpha|^{2/m}}\\
&=:&(e_1\wedge \bar{e}_1)\frac{|e_2|^{2}}{|s_{D}|^{2}}\\
\end{eqnarray*}
where $s_D$ is the $\bQ$-section of $D$ such that $mD=\{s_D^m=0\}$ and we think $|s_D|^{-2}$ as defining a singular Hermitian metric on the $\bQ$-line bundle
$D$. Using the above notations, we can formally write:
\[
{\rm m}_{\delta\vphi}=(e_1\wedge \bar{e}_1)\frac{|e_2|^2}{|s_D|^2} |e_1^{-1}\cdot e_2^{-1}|^2 e^{-\delta\vphi}=:\frac{e^{-\delta\vphi}}{|s_D|^2}.
\]
\begin{rem}
In the even more special case when $X$ is smooth and $D=(1-\beta)E$ for a smooth prime divisor $E$  and $\beta\in (0,1]$. The solution to \eqref{eqcKE} can be shown to be
give rise to a singular K\"{a}hler-Riemannian metric structure $\omega_\vphi:=\sddb\vphi$ with conical singularities along $E$ (see \cite{Don12, JMR15}). In other words away from $E$, $\omega_{\vphi}$ is 
a smooth K\"{a}hler metric, and near the divisor $E$, $\omega_\vphi$ is modeled on the flat conical metric (choosing local coordinate $z=(z_1, z_2, \dots, z_n)$ with $E=\{z_1=0\}$):
\[
\sddb\left(|z_1|^{2\beta}+\sum_{i=2}^n |z_i|^2\right)=\beta^2 \frac{\sqrt{-1} dz_1\wedge d\bar{z}_1}{|z_1|^{2(1-\beta)}}+\sqrt{-1}\sum_{i=2}^{n}dz_i\wedge d\bar{z}_i.
\]
It's easy to verify that, for this model metric, the associated metric tensor can be written as $\left(dr^2+\beta^2 r^2 d\theta^2\right)\times g_{\bC^{n-1}}$ where $r=|z_1|^\beta$.
So $2\pi \beta$ is the metric cone angle. Note that by taking the $-\sddb\log$ on both sides of
\eqref{eqcKE}, we know that the Ricci curvature of $\omega_\vphi$ satisfies the
following equality:
\[
Ric\left(\sddb \vphi\right)=\delta \sddb\vphi+(1-\beta)2\pi \{E\},
\]
where $\{E\}$ denotes the current of integration along $E$. Furthermore, when $\beta=2\pi/m$ for $m\in \bZ_{>0}$, then conical K\"{a}hler-Einstein metrics on the log smooth pair $(X, (1-\beta)E)$ are nothing but orbifold K\"{a}hler-Einstein metrics. Indeed, in these cases the solution to the equation \eqref{eqcKE} can be shown to be orbifold smooth (see \cite{EGZ09, BBEGZ11}). In particular, if $\beta=1$ we are in the case of smooth K\"{a}hler-Einstein metrics. 
\end{rem}

 \subsection{log-Ding-semistability}\label{secKstability}
In this section we recall the definition of log-K-semistability, and its recent equivalence log-Ding-semistability. The concept of log-K-stability was first 
introduced in \cite{Li11} after the formulation the K-stability by Tian and Donaldson. 
\begin{defn}[see {\cite{Tia97, Don02, Li11, LX11, OS15}}]
Let $(X, D)$ be a log-Fano pair. 
\begin{enumerate}
\item For any $\delta>0\in \bQ$ such that $-\delta^{-1}(K_X+D)=L$ is Cartier, a test configuration (resp. a semi test configuration) of $(X, D; L)$ consists of the following data
\begin{itemize}
\item A log pair $(\cX, \cD)$ admitting a $\bC^*$-action and a $\bC^*$-equivariant flat morphism $\pi: \cX\rightarrow \bC^1$, where the $\bC^*$-action on the base $\bC$ is given by the standard multiplication;
\item A $\bC^*$-equivariant $\pi$-ample (resp. $\pi$-semiample) line bundle $\cL$ on $\cX$ such that there is a $\bC^*$-equivariant isomorphism $(\cX, \cD; \cL)|_{\pi^{-1}(\bC\backslash\{0\})}\cong (X, D; -\delta^{-1}(K_X+D))\times \bC^*$.
\end{itemize}
A normal test configuration is called a special test configuration, if the following are satisfied
\begin{itemize}
\item $\cX_0$ is irreducible and normal, and $(\cX_0, \cD_0)$ is a log Fano pair;
\item $\cL = -\delta^{-1}(K_{\cX/\bC}+\cD)$.
\end{itemize}
In this case, we will say that $(X, D)$ specially degenerates to $(\cX_0, \cD_0)$.
\item 
Assume that $(\cX, \cD; \cL)\rightarrow \bC$ is a normal test configuration. Let $\bar{\pi}: (\bar{\cX}, \bar{\cD}; \bar{\cL})\rightarrow \bP^1$ be the
natural equivariant compactification of $(\cX, \cD; \cL)\rightarrow \bC^1$. The log-CM weight of of $(\cX, \cD; \cL)$ is defined by 
\[
\lCM(\cX, \cD; \cL)=\frac{n \delta^{n+1}\bar{\cL}^{n+1}+(n+1) \delta^{n} \bar{\cL}^{n}\cdot (K_{\bar{\cX}/\bP^1}+\cD)}{(n+1) (-K_X-D)^{n}}.
\]
\item
\begin{itemize}
\item The pair $(X, D)$ is called log-K-semistable if $\lCM(\cX, \cD; \cL)\ge 0$ for any normal test configuration $(\cX, \cD; \cL)/\bC^1$ of $(X, D; -\delta^{-1}(K_X+D))$ for any $\delta\in \bQ_{>0}$ (such that $-\delta^{-1}(K_X+L)$ is Cartier).
\item The pair $(X, D)$ is called log-K-polystable if $\lCM(\cX, \cD; \cL)\ge 0$ for any normal test configuration $(\cX, \cD; \cL)/\bC^1$ of $(X, D; -\delta^{-1}(K_X+D))$, and the equality holds if and only if $(\cX, \cD; \cL) \cong (X, D; L) \times\bC^1$.
\end{itemize}
\end{enumerate}
\end{defn}
We will need a related concept of log-Ding-semistability, which is derived from Berman's work in \cite{Bm12}.
\begin{defn}[see \cite{Bm12, Fuj15}]
\begin{enumerate}
\item
Let $(\cX, \cD; \cL)/\bC^1$ be a normal semi-test configuration of $(X, D; -\delta^{-1}(K_X+D)$ and $(\bar{\cX}, \bar{\cD}; \bar{\cL})/\bP^1$ be its natural
compactification. Let $\Delta_{(\cX,\cD; \cL)}$ be the $\bQ$-divisor on $\cX$ satisfying the following conditions:
\begin{itemize}
\item The support ${\rm Supp}\; \Delta_{(\cX, \cD; \cL)}$ is contained in $\cX_0$;
\item The divisor $\delta^{-1} \Delta_{(\cX, \cD; \cL)}$ is a $\bZ$-divisor corresponding to the divisorial sheaf $\bar{\cL}(\delta^{-1}(K_{\bar{\cX}/\bP^1}+\bar{\cD}))$.
\end{itemize}
\item 
The log-Ding invariant $\lDing(\cX, \cD; \cL)$ of $(\cX, \cD; \cL)/\bC$ is defined as:
\[
\lDing((\cX, \cD); \cL):=\frac{-\delta^{n+1} \bar{\cL}^{n+1}}{(n+1) (-K_X-D)^n}-\left(1-\lct(\cX, \cD-\Delta_{(\cX, \cD; \cL)}; \cX_0)\right).
\]
\item
$(X, D)$ is called log-Ding semistable if $\lDing(\cX, \cD; \cL)\ge 0$ for any normal test configuration $(\cX, \cD; \cL)/\bC$ of $(X, D; -\delta^{-1}K_X)$.
\end{enumerate}
\end{defn}
\begin{rem}
For special test configurations, 
log-CM weight coincides with log-Ding invariant:
\begin{equation}\label{CMstc}
\lCM(\cX, \cD; \cL)=\frac{-\delta^{n+1} \bar{\cL}^{n+1}}{(n+1)(-K_X-D)^n}=\lDing(\cX, \cD; \cL).
\end{equation}
By the recent work of \cite{BBJ15}, (log-)Ding-semistability is equivalent to (log-)K-semistability. 
Notice that since the work in \cite{LX11}, it was known that to test log-K-semistability (or log-K-polystability), one only needs to test on special test configurations.
\end{rem}

\subsection{Filtrations and Fujita's result}

We recall the relevant definitions about filtrations after \cite{BC11} (see also \cite{BHJ15} and \cite[Section 4.1]{Fuj15}).
\begin{defn}\label{defil}
 A {\it good filtration} of a graded $\bC$-algebra $S=\bigoplus_{m=0}^{+\infty}S_m$ is a decreasing, left continuous, multiplicative
and linearly bounded $\bR$-filtrations of $S$. In other words, for each $m\ge 0\in \bZ$, there is a family of subspaces $\{\cF^{x} S_m\}_{x\in \bR}$ of $S_m$ such that:
\begin{enumerate}
\item $\cF^x S_m\subseteq \cF^{x'}S_m$, if $x\ge x'$;
\item $\cF^x S_m=\bigcap_{x'<x}\cF^{x'} S_m$; 
\item $\cF^x S_m\cdot \cF^{x'} S_{m'}\subseteq \cF^{x+x'} S_{m+m'}$, for any $x, x'\in \bR$ and $m, m'\in \bZ_{\ge 0}$;
\item $e_{\min}(\cF)>-\infty$ and $e_{\max}(\cF)<+\infty$, where $e_{\min}(\cF)$ and $e_{\max}(\cF)$ are defined by the following 
operations:

\begin{equation}
\def\arraystretch{1.5}
\begin{array}{l}
e_{\min}(S_m,\cF)=\inf\{t\in\bR; \cF^t S_m\neq S_m\}; \\
e_{\max}(S_m,\cF)=\sup\{t\in\bR; \cF^t S_m\neq 0\};\\
\displaystyle e_{\min}(\cF)=e_{\min}(S_{\bullet}, \cF)=\liminf_{m\rightarrow +\infty} \frac{e_{\min}(S_m, \cF)}{m}; \\
\displaystyle e_{\max}(\cF)=e_{\max}(S_{\bullet}, \cF)=\limsup_{i\rightarrow +\infty} \frac{e_{\max}(S_m, \cF)}{m}. 
\end{array}
\end{equation}

\end{enumerate}
\end{defn}

Define $S^{(t)}=\bigoplus_{k=0}^{+\infty} \cF^{kt} S_k$. When we want to emphasize the dependence of $ S^{(t)}$ on the filtration $\cF$, we also denote $ S^{(t)}$ by $\cF S^{(t)}$.
The following concept of volume will be important for us:
\begin{equation}
\vol\left( S^{(t)}\right)=\vol\left(\cF S^{(t)}\right):=\limsup_{k\rightarrow+\infty}\frac{\dim_{\bC}\cF^{mt}S_m}{m^{n}/n!}.
\end{equation}

Now assume $L$ is an ample line bundle over $X$ and $S=\bigoplus_{m=0}^{+\infty} H^0(X, L^m)=\bigoplus_{m=0}^{+\infty} S_m$ is the section ring of $(X, L)$. Then following \cite{Fuj15}, we can define a sequence of ideal sheaves on $X$:
\begin{equation}\label{filtideal}
I^{\cF}_{(m,x)}={\rm Image}\left(\cF^xS_m\otimes L^{-m}\rightarrow \cO_X\right), 
\end{equation}
and define $\overline{\cF}^x S_m:=H^0(V, L^m\cdot I^{\cF}_{(m,x)})$ to be the saturation of $\cF^x S_m$ such that
$\cF^x S^m\subseteq \overline{\cF}^x S_m$.
$\cF$ is called saturated if $\overline{\cF}^x S_m=\cF^x S_m$ for any $x\in \bR$ and $m\in \bZ_{\ge 0}$. 
Notice that with our notations we have:
\[
\vol\left(\overline{\cF} S^{(t)}\right):=\limsup_{k\rightarrow+\infty}\dim_{\bC}\frac{\overline{\cF}^{kt}H^0(X, kL)}{k^n/n!}.
\]
The following result was proved by Fujita by applying a criterion for Ding semistability (\cite[Proposition 3.5]{Fuj15}) to a sequence of semi-test configurations constructed from a filtration.
We will state a log version of it in Section \ref{seclogFujita}.
\begin{thm}[{\cite[Theorem 4.9]{Fuj15}}]\label{Fujthm}
Assume $(X, -K_X)$ is Ding-semistable. Let $\cF$ be a {\it good} graded filtration of $S=\bigoplus_{m=0}^{+\infty} H^0(X, mL)$ where $L=-\delta^{-1}K_X$. 
Then the pair $(X \times\bC, \cI_\bullet^\delta \cdot (t)^{d_\infty})$ is sub log canonical, where
\begin{eqnarray*}
&&\mathcal{I}_m=I^{\cF}_{(m, m e_{+})}+I^{\cF}_{(m, m e_{+}-1)}t^1+\cdots+I^{\cF}_{(m, m e_{-}+1)} t^{m(e_+-e_-)-1}+(t^{m(e_+-e_-)}),\\
&&d_{\infty}=1-\delta(e_+-e_-)+\frac{\delta^{n+1}}{((-K_X)^n)}\int_{e_-}^{e_+}\vol\left(\overline{\cF} S^{(t)}\right)dt, \\
\end{eqnarray*}
and $e_+, e_-\in \bZ$ with $e_+\ge e_{\max}(S_\bullet, \cF)$ and $e_-\le e_{\min}(S_\bullet, \cF)$.
\end{thm}

\subsection{Normalized volumes of valuations}

In this section, we briefly recall the concept of valuations and their normalized volumes. Let $Y={\rm Spec}(R)$ be a normal affine variety. A real valuation $v$ on the function field $\bC(Y)$ is a map $v: \bC(Y)\rightarrow \bR$, satisfying:
\begin{enumerate}
\item
$v(fg)=v(f)+v(g)$;
\item $v(f+g)\ge \min\{v(f), v(g)\}$.
\end{enumerate}
In this paper we also require $v(\bC^*)=0$, i.e. $v$ is trivial on $\bC$. Denote by $\cO_v=\{f\in \bC(Y); v(f)\ge 0\}$ the valuation ring of $v$. The valuation $v$ is said to be finite over $R$, or on $Y$, if $\cO_v\supset R$. Let $\Val_Y$ denote the space of all real valuations which are trivial on $\bC$ and finite over $R$. 
Fix a closed point $o\in Y$ with the corresponding maximal ideal of $R$ denoted by $\fm$. We will be interested in the space $\Val_{Y,o}$ of all valuations $v$ with ${\rm center}_Y(v)=o$. If $v\in \Val_{Y,o}$, then $v$ is centered on the local ring $R_\fm$ (see \cite{ELS03}). In other words, $v$ is nonnegative on $R_\fm$ and is strictly positive on the maximal ideal of the local ring $R_\fm$.
For any $v\in \Val_{Y,o}$, we define its valuative ideals:
\[
\fa_p(v)=\{f\in R; v(f)\ge p\}.
\]
The by \cite[Proposition 1.5]{ELS03}, $\fa_p(v)$ is $\fm$-primary, and so is of finite codimension in $R$ (cf. \cite{AM69}). We define the volume of $v$ as the limit:
\[
\vol(v)=\limsup_{p\rightarrow +\infty}\frac{\dim_{\bC}R/\fa_p(v)}{p^n/n!}.
\]
By \cite{ELS03}, \cite{Mus02} and \cite{Cut12}, the limitsup on the right hand side is a limit and is equal to the multiplicity of the graded family of ideals $\fa_\bullet=\{\fa_p\}$:
\begin{equation}\label{vol=mul}
\vol(v)=\lim_{p\rightarrow +\infty}\frac{e(\fa_p)}{p^n}=:e(\fa_\bullet),
\end{equation}
where $e(\fa_p)$ is the Hilbert-Samuel multiplicity of $\fa_p$.

From now on in this paper, we assume that $Y$ has $\bQ$-Gorenstein klt singularities. Following \cite{JM10} and \cite{BFFU13}, we can define the log discrepancy for any valuation $v\in \Val_Y$. This is achieved in three steps in \cite{JM10} and \cite{BFFU13}. Firstly, for a divisorial valuation $\ord_E$ associated to a prime divisor $E$ over $Y$, define $A_Y(E)=\ord_{E}(K_{Z/Y})+1$, where $\pi: Z\rightarrow Y$ is a smooth model of $Y$ containing $E$. Next for any
quasimonomial valuation (also called Abhyankar valuation) $v\in {\rm QM}_\eta(Z,D)$ where $(Z,D=\sum_{k=1}^N D_k)$ is log smooth and $\eta$ is a generic point of an irreducible component of $D_1\cap \dots\cap D_N$, we define $A_Y(v)=\sum_{k=1}^N v(D_k)A_Y(D_k)$. Lastly for any valuation $v\in \Val_Y$, we define 
\[
A_Y(v)=\sup_{(Z,D)}A_Y(r_{Z,D}(v))
\]
where $(Z, D)$ ranges over all log smooth models over $Y$, and $r_{(Z,D)}: \Val_Y\rightarrow {\rm QM}(Z,D)$ are contraction maps that induce a homeomorphism $\displaystyle\Val_Y\rightarrow \lim_{\stackrel{\longleftarrow}{(Z,D)}}{\rm QM}(Z,D)$. For details, see \cite{JM10} and \cite[Theorem 3.1]{BFFU13}. A basic property of $A_Y$ is that for any proper birational morphism $Z\rightarrow Y$, we have (see \cite[Remark 5.6]{JM10}, \cite[Proof 3.1]{BFFU13}):
\begin{equation}\label{ldbirat}
A_{Y}(v)=A_Z(v)+v(K_{Z/Y}).
\end{equation}
Now we can define the normalized volume for any $v\in \Val_{Y,o}$:
\[
\hvol(v):=\hvol_Y(v)=\left\{
\begin{array}{lll}
A_Y(v)^n\vol(v), & \mbox{ if } & A_Y(v)<+\infty,\\
+\infty, & \mbox{ if } & A_Y(v)=+\infty.
\end{array}
\right.
\]
Notice that $\hvol(v)$ is rescaling invariant: $\hvol(\lambda v)=\hvol(v)$ for any $\lambda>0$. By Izumi's theorem (see \cite{Izu85,Ree89,ELS03,BFJ12,JM10,Li15a}), one can show that
$\hvol$ is uniformly bounded from below by a positive number on $\Val_{Y,o}$. If $\hvol$ has a global minimizer $v_*$ on $\Val_{Y,o}$, i.e. $\hvol(v_*)=\inf_{v\in \Val_{Y,o}}\hvol(v)$, then
we will say that {\it $\hvol$ is globally minimized at $v_*$ over $(Y,o)$}. In \cite{Li15a}, the first author conjectured that this holds for all $\bQ$-Gorenstein klt singularities and proved that this is the case under a semi-continuity hypothesis. 

We will also need a log version of $\hvol$. For this, we 
assume $(Y, E)$ is a log pair such that $E$ is an effective $\bQ$-Weil divisor, $K_Y+E$ is 
$\bQ$-Cartier and $(Y, E)$ has klt singularities. Then we can
follow the above definition by replacing $A_Y(v)$ by the log
discrepancy $A_{(Y, E)}(v)$ of $v$ with respect to $(Y, E)$. 
More precisely, for a divisorial valuation $\ord_F$ associated
to a prime divisor $F$ over $Y$, define $A_{(Y,E)}(F)=\ord_F(K_{Z/(Y,E)})+1$,
where $\pi: Z\to Y$ is a smooth model of $Y$ containing $F$ and
$K_{Z/(Y,E)}:= K_Z-\pi^*(K_Y+E)$. Then we may extend the log discrepancy
function $A_{(Y,E)}(\cdot)$ to all valuations $v\in\Val_Y$ in the same way as 
in the absolute case $A_{Y}(\cdot)$. Indeed we can fix a smooth model $\phi: W\rightarrow Y$ and write:
\[
K_W=\phi^*(K_Y+E)+G.
\]
Notice that $G$ is $\bQ$-Cartier. Because $W$ is smooth and $\bC(W)\cong \bC(Y)$, we can define, for any $v\in \Val_Y\cong \Val_W$,
\[
A_{(Y, E)}(v):=A_W(v)+v(G),
\]
where $v(G):=\frac{1}{k}v(\cI_{kG})$ if $kG$ is Cartier with the ideal sheaf $\cI_{kG}$ (see \cite{dFH09}).
With the above definition, we define for any $v\in \Val_{Y,o}$:
\[
\hvol_{(Y,E)}(v)=\left\{
\begin{array}{lll}
A_{(Y,E)}(v)^n\vol(v), & \mbox{ if } & A_{(Y,E)}(v)<+\infty,\\
+\infty, & \mbox{ if } & A_{(Y,E)}(v)=+\infty.
\end{array}
\right.
\]
Notice that in the case when $K_Y$ and $E$ are $\bQ$-Cartier, then we have:
\begin{equation}\label{pldvsld}
A_{(Y,E)}(v)=A_Y(v)-v(E).
\end{equation}
In particular $A_Y(v)=\infty$ if and only if $A_{(Y,E)}(v)=\infty$. 


\section{Log-K-polystability of projective cones}\label{seclogK}

\subsection{Berman's result on log-K-polystability}

In this section, we state and sketch a proof of Berman result on 
log-Ding-polystability which will be shown to imply a log version 
of Fujita's result in Section \ref{seclogFujita}.
\begin{thm}[\cite{Bm12}]\label{logBerm}
If $(X, D; L)$ has a conical K\"{a}hler-Einstein potential, then $(X,D)$ is log-K-polystable. In particular, it is log-K-semistable, and equivalently log-Ding-semistable.

\end{thm}
This was proved by Berman (see \cite[Section 4.3]{Bm12}). For the reader's convenience, we provide a sketch of the proof. 
\begin{proof}
Define the space of locally bounded positively curved Hermitian metrics on $L$:
\[
\cH_{\infty}(X,L):=\{e^{-\vphi} \;|\; \vphi=\{\vphi_j\} \in L^\infty_{\rm loc} (X) \text{ is l.s.c. and } \sddb\vphi\ge 0\},
\]
where $\omega_{\vphi}:=\sddb\vphi\ge 0$ is in the sense of currents.  We fix a locally bounded reference Hermitian metric $e^{-\phi_0}$ on $L$ (see section \ref{seconic}). For any $\vphi\in \cH_\infty(X, h)$ define the log-Ding-energy (see \cite{Din88}) as:
\begin{equation}
F(\vphi)=F^{0}_{\phi_0}(\vphi)- \log\left(\frac{1}{L^{n}}\int_X e^{-\delta\vphi} \right).
\end{equation}
Here we used the notations in Section \ref{seconic} and the Monge-Amp\`{e}re energy $F^0_{\phi_0}(\vphi)$ is defined as:
\[
F^0_{\phi_0}(\vphi)=-\frac{\delta}{(n+1)}\sum_{k=0}^{n}\frac{1}{L^n}\int_X (\vphi-\phi_0) \; (\sddb\vphi)^k\wedge (\sddb\phi_0)^{n-k}.
\]
The reason for considering the log-Ding-energy is that its critical point is exactly the conical K\"{a}hler-Einstein potential, i.e. the solution (up to a constant) to the complex Monge-Amp\`{e}re equation:
\begin{equation}\label{CMA}
(\sddb\vphi)^{n}={\rm m}_{\delta\vphi}.
\end{equation}
Moreover if there is a conical K\"{a}hler-Einstein potential $e^{-\phi_{\rm cKE}}$ on the pair $(X, D)$, then $F(\vphi)$ is bounded from below on $\cH_\infty(X, L)$ by $F(\phi_{\rm cKE})<+\infty$
(see \cite{Bn15, BBEGZ11}). We will later apply this fact to a family of Hermtian metrics coming from any test configuration.

Now suppose $((\cX, \cD); p\cL)$ is a normal test configuration of $((X,D); pL)$. 
Let $e^{-\Phi}$ be any locally bounded Hermitian metric on $p \cL$ which is positively curved, i.e. satisfying $\sddb\Phi\ge 0$ in the sense of current. Following the approach in \cite{Bm12}, we consider the divisor $\Delta=(K_{\cX/\bC^1}+\cD)+\delta \cL$ with its defining section $\tS$. 
Then we get a relative measure on $\cX$:
\[
v_{\Phi}:=|\tS|^2 {\rm m}_{\Phi}.
\]
Here as before, ${\rm m}_{\Phi}$ is a globally defined measure on $\cX$, which on the regular part of $\cX$ is equal to $\frac{\exp\left(-\frac{\delta}{p}\Phi\right)}{|\mathcal{S}_\cD|^2}$ where $\cS_\cD$ is the defining section of $\cD$ on the regular part of $\cX$. Notice that it has a positive curvature current:
\[
\sddb\left(\frac{\delta}{p}\Phi+\log|S_\cD|^2\right)=\left(\frac{\delta}{p}\sddb\Phi\right)+\{S_\cD=0\}
\]
where $\{S_\cD=0\}$ is the current of integration along the divisor $\cD$. By Berndtsson's subharmonicity theorem (see \cite{Bn15, BP08}), we know that the function on the base:
\[
G(t):=-\log \left(\int_{\cX/\bC}v_{\Phi}\right)
\]
is subharmonic. Notice that if $e^{-\Phi}$ is $\bC^*$-invariant which is the case for the metrics we will choose, the subharmonicity implies $G(t)$ is convex in $-\log|t|$.

We can argue as Berman in \cite[Section 3.3]{Bm12} to show that the Lelong number of $G(t)$ is equal to:
\begin{equation}\label{logLelong}
l_0:=1-\lct((\cX, \cD-\Delta); \cX_0).
\end{equation}
Recall that the Lelong number of $G(t)$ can be defined as:
\begin{equation}\label{defLelong}
l_0=\liminf_{t\rightarrow 0}\frac{G(t)}{\log|t|}.
\end{equation}
On the other hand, by the equivariant isomorphism:
\[
\mu: (\cX|_{\bC^*}, p\cL) \cong (X, pL)\times \bC^*
\]
we get a family of Hermitian metrics $\vphi_t\in \cH_\infty(X, L)$ such that 
\[
e^{-\vphi_t}=\left.(\mu^{-1})^* e^{-\frac{1}{p}\Phi}\right|_{X\times\{t\}}.
\] 
From the construction, we can see that:
\begin{equation}\label{Dingdec}
F(\vphi_t)=F^0_{\phi_0}(\vphi_t)-G(t).
\end{equation}
Now there are two natural choices of the locally bounded Hermitian metric on $p\cL$ for which we can carry out the above constructions to get two families of Hermtian metrics on $L\rightarrow X$. For the first one we use the well known fact that there is an $\bC^*$-equivariant embedding:
\[
\alpha_1: \cX\rightarrow \bP^N\times\bC^1 \text{ with } p\cL= \alpha_1^* \cO_{\bP^N}(1)
\]
where $\bC^*$ acts on $\bC^1$ by multiplication and acts on $\bP^N$ through a one-parameter subgroup of ${\rm PGL}(N+1, \bC)$. Pulling back the Fubini-Study metric on $\cO_{\bP^N}(1)$, we get a smooth Hermitian metric, denoted by $e^{-\Phi^{(1)}}$, on $p\cL$. The family of Hermitian metrics $\{\vphi^{(1)}_t\}$ associated to $e^{-\Phi^{(1)}}$ as above is usually called Bergman metrics on $L\rightarrow X$. 
We can prove the log-Ding-semistability by using family. Indeed, it's well known that $F^0_{\phi_0}(\vphi^{(1)}_t)$ is concave in $-\log |t|$, and its slope at infinity is equal to:
\begin{equation}\label{F0dir}
\lim_{t\rightarrow 0}\frac{F^0_{\phi_0}(\vphi^{(1)}_t)}{-\log|t|}=-\frac{\delta\bar{\cL}^{n+1}}{L^{n}}=-\frac{\delta^{n+1}\bar{\cL}^{n+1}}{(-K_X-D)^n}.
\end{equation}
So combining this formula \eqref{logLelong}-\eqref{Dingdec}, we get:
\begin{equation}\label{expansion1}
\lim_{t\rightarrow 0}\frac{F(\vphi^{(1)}_t)}{-\log|t|}=-\frac{\delta^{n+1}\bar{\cL}^{n+1}}{(-K_X-D)^{n}}-(1-\lct((\cX, \cD-\Delta); \cX_0).
\end{equation}
As mentioned above, if $(X, D)$ is conical K\"{a}hler-Einstein, then $F(\vphi^{(1)}_t)$ is bounded from below, so the right-hand-side of \eqref{expansion1} is non-negative. 

To prove the stronger polystability, we need to consider another locally bounded metric, which is obtained by solving a homogeneous complex Monge-Amp\`{e}re equation on $\cX|_{\{|t|\le 1\}}$ 
\[
\left\{
\begin{array}{l}
(\sddb \Phi^{(1)}+\sddb \Psi)^{n+1}=0;\\
\Psi|_{X\times S^1}=0. 
\end{array}
\right.
\]
It's now well known that there exists a bounded solution $\Psi$. Denoting $e^{-\Phi^{(2)}}=e^{-(\Phi^{(1)}+\Psi)}$, then the family of Hermitian metrics $\{\vphi_t^{(2)}\}$ associated to the locally bounded Hermitian metric $e^{-\Phi^{(2)}}$ is a so-called weak geodesic ray in $\cH_{\infty}(X, L)$ emanating from $e^{-\vphi_1}$. 
If we choose $\phi_0$ to be the conical K\"{a}hler-Einstein potential, then we have 
\[
\left.\frac{d}{d(-\log |t|)}\right|_{|t|=1^{-}}F(\vphi^{(2)}_t))\ge 0, \text{ and } \lim_{t\rightarrow 0}\frac{d}{d(-\log|t|)}F(\vphi^{(2)}_t)=0. 
\]
By the convexity of $F(\vphi^{(2)}_t)$ with respect to $-\log|t|$, 
this implies that $F(\vphi^{(2)}_t)$ is affine with respect to $-\log|t|$. It's well known that for the geodesic ray $\{\vphi^{(2)}_t\}$, $F^{0}_{\phi_0}(\vphi^{(2)}_t)$ is affine in the variable $-\log|t|$, so $G(t)$ must be affine with respect to $-\log|t|$. 
By arguing as in \cite{Bm12} (see also \cite{Bn15}), we can conclude that $(\cX, \cD; \cL)$ is indeed a product test configuration.

\end{proof}

\subsection{cKE metric on the projective cone over a cKE variety}\label{secprojcone}
Let $(V, E)$ be a log Fano pair. Assume $H=-r^{-1}(K_V+E)$ is an ample Cartier divisor. Consider the
projective cone over $V$ with polarization $H$:
\[
\ocC={\rm Proj}\left(\bigoplus_{m=0}^{+\infty} S_m\right)
\]
where 
\[
S_m=\bigoplus_{k=0}^{m }\left(H^0(V, kH)\cdot u^{m-k}\right).
\]
Geometrically, $\ocC=\cC\cup V_\infty$ where $\cC$ is the affine cone over $V$ defined as
\[
\cC={\rm Spec}\left(\bigoplus_{m=0}^{+\infty}H^0(V, mH)\right); 
\]
$V_\infty$ is the divisor at infinity, which is clearly isomorphic to $V$. Let $\cE$ be the corresponding
divisor on $\ocC$. By \cite[Section 3.1]{Kol13}, $K_{\ocC}+\cE$ is $\bQ$-Cartier, $(\ocC, \cE)$ has klt singularities and moreover:
\begin{equation}\label{logac}
-(K_{\ocC}+\cE)\sim_{\bQ} (1+r)V_\infty.
\end{equation}
Define the cone angle parameter:
\begin{equation}\label{angle}
\beta=\frac{r}{n}.
\end{equation}
The following lemma is well known. For the absolute case, see \cite[Corollary 2.1.13]{IP99}.
\begin{lem}\label{fanoindex}
 Assume that $(V^{n-1},E)$ is a log Fano pair. If $H=-r^{-1}(K_V+E)$
 is Cartier for $r\in\bQ_{>0}$, then $0<r\leq n$.
\end{lem}

\begin{proof}
 The proof is the same as in \cite[Corollary 2.1.13]{IP99}. By the 
 general Kawamata-Viehweg vanishing (see \cite[2.16]{Kol97}), 
 $H^i(V,\cO_V(mH))=0$ for any $i>0$ and $m>-r$. By Riemann-Roch,
 $\chi(V, mH)$ is a polynomial in $m$ of degree $n-1$. By the above vanishing,
 we know that this polynomial has roots $m=-1,-2,\cdots,1-\lceil r\rceil$.
 So we get $r\leq n$.
\end{proof}

The above lemma yields $0<r\le n$, so $\beta\in (0,1]$. 
Then it's easy to verify that the pair $(\ocC, \cE+(1-\beta)V_\infty)$ is klt, whose anti-log-canonical divisor
is given as:
\begin{equation}\label{Lcone}
-(K_{\ocC}+\cE+(1-\beta) V_\infty)\sim_{\bQ} \left((1+r)-(1-\frac{r}{n})\right)V_\infty=r\frac{n+1}{n}V_\infty.
\end{equation}
We will denote by $L$ the line bundle associated to the Cartier divisor $V_\infty$. Notice that $\left.L\right|_{V_\infty}\cong H$ so that
$(L^n)=(H^{n-1})$.
\begin{prop}\label{cKEproj}
If $(V, E)$ has a conical K\"{a}hler-Einstein Hermitian metric on $(V, E; H)$, then there exists a conical KE potential on $(\ocC, \cE+(1-\beta)V_\infty; L)$.
\end{prop}
\begin{proof}
Let $h:=e^{-\phi}$ be the cKE potential on $(V, E; H)$. By the calculation in \cite{Li13}, the 
cKE potential $e^{-\vphi_L}$ on $(\ocC, \cE+(1-\beta)V_\infty; L)$ should satisfy the equality:
\begin{equation}\label{cKE}
\|s_{V_\infty}\|_{h_L}^2:=|s_{V_\infty}|^2 e^{-\vphi_L}=\frac{h}{\left(1+h^{r/n}\right)^{n/r}},
\end{equation}
where $s_{V_\infty}$ is the defining section of $\cO_\ocC(V_\infty)$ and $h$ is viewed as a non-negative function on $\ocC\setminus\{o\}$. 
 \begin{rem}
If $E=\emptyset$ this is just an appropriately normalized form of the result in \cite[Lemma 3]{Li13}.
The expression in \eqref{cKE} is clearly a generalization of the classical Fubini-Study metric on $\cO_{\bP^n}(1)$.
\end{rem}
It's easy to see that the formula \eqref{cKE} defines a locally bounded Hermitian metric $e^{-\vphi_L}$ on $L$. We will directly verify that this is indeed a 
cKE potential. We will denote by $\bar{f}: \ocC\setminus\{o\}\cong H \rightarrow V_\infty=V$ the natural projection. It's also convenient to denote $\rho=\log h$ such that 
\begin{itemize}
\item $h=e^\rho$ and $\sddb \rho=-\bar{f}^*\omega+2\pi \{V_\infty\}$, where $\omega$ is the curvature current of the cKE potential $h$ over $(V, E; H)$ and $\{V_\infty\}$ denotes the current of integration along $V_\infty$;
\item $\partial\rho=\partial\log a+\frac{d\xi}{\xi}$. As a consequence we have:
\[
\bar{f}^*\omega^{n-1}\wedge \partial\rho\wedge \bar{\partial}\rho=\bar{f}^*\omega^{n-1}\wedge \frac{d\xi\wedge d\bar{\xi}}{|\xi|^2}
\]
\end{itemize}
Using the above notation, we can calculate:
\begin{eqnarray*}
-\sddb\log\|s\|_{h_L}^2&=&(-\sddb\log\rho)+\frac{n}{r}\sddb\log(1+e^{\beta\rho})\\
&=&\omega-2\pi \{V_\infty\}+\frac{n}{r}\left(-\frac{\beta e^{\beta\rho}}{1+e^{\beta\rho}}\omega+\frac{\beta^2 e^{\beta\rho}}{(1+e^{\beta\rho})^2} \partial \rho\wedge \bar{\partial}\rho \right)\\
&=&\frac{\omega}{1+e^{\beta\rho}}+\frac{\beta e^{\beta\rho}}{(1+e^{\beta\rho})^2}\partial\rho\wedge \bar{\partial}\rho-2\pi\{V_\infty\}.
\end{eqnarray*}
Here we identify $\omega$ on $V$ with its pull back $\bar{f}^*\omega$ to $\ocC\setminus\{o\}$. So we get the curvature current of $e^{-\vphi_L}$:
\begin{equation}\label{curvproj}
\sddb\vphi_L=\frac{\omega}{1+e^{\beta\rho}}+\frac{\beta e^{\beta\rho}}{(1+e^{\beta\rho})^2}\partial\rho\wedge\bar{\partial}\rho. 
\end{equation}
This is clearly a positive current over $\ocC\setminus \{o\}$. The volume form (or the Monge-Amp\`{e}re measure) of $\sddb\vphi_L$ is equal to 
\begin{eqnarray*}
(\sddb\vphi_L)^n&=&n\beta\frac{e^{\beta\rho}}{(1+e^{\beta\rho})^{n+1}}\bar{f}^*\omega^{n-1}\wedge \partial\rho\wedge \bar{\partial}\rho\\
&=&n\beta \frac{e^{\beta \rho}}{(1+e^{\beta\rho})^{n+1}}\bar{f}^*\omega^{n-1}\wedge \frac{d\xi\wedge d\bar{\xi}}{|\xi|^2}.
\end{eqnarray*}
Because $e^{-\phi}$ is conical K\"{a}hler-Einstein on $(V, E; H=-r^{-1}(K_V+E))$, it satisfies the following equation on the regular part of $V$:
\begin{equation}\label{cKEbase}
\omega^{n-1}=\frac{e^{-r\phi}}{|s_E |^2},
\end{equation}
where $s_E$ is the defining section of $\cO_V(E|_{V_{\rm reg}})$.
Taking the curvature on both sides, we get:
\begin{equation}\label{cKEV}
Ric(\omega^{n-1})=\sddb(r\phi+\log |s_E|^2)=r\omega+2\pi \{E\}.
\end{equation}
We define a measure on $\ocC_{\rm reg}$:
\begin{eqnarray*}
{\bf M}&:=&(\sddb\vphi_L)^n\cdot |s_{V_\infty}|^{2(1-\beta)} |s_\cE|^2\\
&=&n\beta \frac{e^{\beta\rho}}{\left(1+e^{\beta\rho}\right)^{n+1}}\bar{f}^*\omega^{n-1}\wedge \partial \rho\wedge \bar{\partial}\rho\cdot |s_{V_\infty}|^{2(1-\beta)}|s_\cE|^2\\
&=&n\beta \frac{e^{\beta\rho}|s_{V_\infty}|^{2(1-\beta)}}{\left(1+e^{\beta\rho}\right)^{n+1}}\frac{d\xi\wedge d\bar{\xi}}{|\xi|^2}\wedge \bar{f}^*\left(\omega^{n-1} \cdot |s_E|^2\right)\\
\end{eqnarray*}
Using the cKE equation \eqref{cKEbase}, we know that ${\bf M}$ is induced given by a bounded (indeed continuous) volume form on the regular part of $\ocC\setminus\{o\}$, which induces a locally 
bounded Hermitian metric on $-(K_\ocC+(1-\beta)V_\infty+\cE)$: 
\begin{eqnarray*}
e^{-\psi} 
&:=&n\beta \frac{e^{\beta\rho}|s_{V_\infty}|^{2(1-\beta)}}{\left(1+e^{\beta\rho}\right)^{n+1}}\frac{d\xi\wedge d\bar{\xi}}{|\xi|^2}\wedge \bar{f}^*\left(\omega^{n-1} \cdot |s_E|^2\right)\\
&=&n\beta \frac{e^{\beta\rho}|s_{V_\infty}|^{2(1-\beta)}}{\left(1+e^{\beta\rho}\right)^{n+1}}\frac{d\xi\wedge d\bar{\xi}}{|\xi|^2}\wedge \bar{f}^*\left(e^{-r\phi}\right).
\end{eqnarray*}
We can easily calculate the curvature current of $e^{-\psi}$:
\begin{eqnarray*}
\def\arraystretch{2}
\sddb\psi
&=&(-\beta \sddb \rho+2\pi\beta\{V_\infty\})+(n+1)\sddb \log(1+e^{\beta \rho})+\bar{f}^*\left(r\omega\right)\\
&=&\beta\omega+r \omega+(n+1)\left(-\frac{\beta e^{\beta\rho}}{1+e^{\beta\rho}}\omega+\frac{\beta^2 e^{\beta\rho}}{(1+e^{\beta\rho})^2} \partial \rho\wedge \bar{\partial}\rho \right)\\
&=&(n+1)\beta\left(\frac{\omega}{1+e^{\beta\rho}}+\frac{\beta e^{\beta\rho}}{(1+e^{\beta\rho})^2}\partial\rho\wedge\bar{\partial}\rho\right)=\frac{(n+1)r}{n}\sddb\vphi_L.
\end{eqnarray*}
Compared to the identity \eqref{Lcone}, the above equality means that the Hermitian metric $e^{-\psi}$ on $-(K_{\ocC}+\cE+(1-\beta)V_\infty)$ is equal to $e^{-r\frac{n+1}{n}\vphi_L}$ up to positive rescaling. 
Moreover, combining the above equalities, we see that $e^{-\vphi_L}$ satisfies the following equation on the regular part of $\ocC$:
\begin{eqnarray*}
(\sddb\vphi_L)^n&=&\frac{e^{-r\frac{n+1}{n}\vphi_L}}{|s_{V_\infty}|^{2(1-\beta)}|s_\cE|^2}.
\end{eqnarray*}
This identity holds a priori on $\ocC_{\rm reg}$. However, it extends to hold on the whole $\ocC$ in the pluripotential sense because both sides does not charge the singular set $\ocC_{\rm sing}$. Indeed, the latter follows from the fact $\vphi_L$ is locally bounded and $(\ocC, (1-\beta)V_\infty+\cE)$ is klt.

So by the definition in Section \ref{seconic}, we see that $e^{-\vphi_L}$ is indeed a conical K\"{a}hler-Einstein potential on $(\ocC, \cE+(1-\beta)V_\infty; L)$.

\end{proof}

\subsection{An orbifold version}
For the notions used in the following proposition, see Section \ref{AppSeif}. 

\begin{prop}\label{propcKEorb}
Let $f: Y^n\rightarrow (V^{n-1}, \Delta=\sum_{i} (1-\frac{1}{m_i})D_i)$ be a Seifert $\bC^*$-bundle. Assume $c_1(Y/V)$ is positive and let $\ocC_{\orb}$ denote the associated orbifold projective cone. Assume that $(V, \Delta)$ is orbifold smooth and the following two conditions are satisfied:
\begin{enumerate}
\item
The orbifold canonical class $K_{\orb}:=K_V+\Delta$ is anti ample and there is a smooth orbifold K\"{a}hler-Einstein metric on $(V, \Delta; -K_{\orb})$
\item
The Chern class $c_1(Y/V)=-r^{-1} (K_V+\Delta)$ with $0<r\le n$.
\end{enumerate} 
Then there exists a conical K\"{a}hler-Einstein potential on $(\ocC_{\orb}, (1-\beta)V_\infty; \cO_{\ocC_\orb}(V_\infty))$ with $\beta=\frac{r}{n}$.
\end{prop}
\begin{proof}
For simplicity, we will just denote $\ocC_\orb$ by $\ocC$ and by $\ocC^\circ$ the punctured projective cone $\ocC\setminus\{o\}$.  Assume $\omega$ is a smooth orbifold K\"{a}hler metric on $(V, \Delta)$ in the orbifold first Chern class $c_1(H_{\orb})=-r^{-1}c_1(K_{\orb})$. Then there exists a smooth orbifold Hermitian metric $e^{-\phi}=\{e^{-\phi_i}\}$ on $H_{\orb}$ whose orbifold Chern curvature is $\omega$. In other words, for any local orbifold chart $\pi_i: W_i\rightarrow W_i/\mu_m\cong U_i$, $e^{-\wtd{\phi}}:=\pi_i^*e^{-\phi_i}$ is a smooth Hermitian metric on the line bundle $\pi_i^*H_{\orb}$, whose Chern curvature $\wtd{\omega}:=\sddb \pi_i^* \phi_i=\pi_i^*\omega$ is a smooth K\"{a}hler metric on $W_i$.

Following \eqref{cKE}, we define a locally bounded Hermitian metric on the $\bQ$-line bundle $\bar{f}^* H_\orb=V_\infty$ by first defining $\mu_m$-invariant functions over $W_i\times \bC$:
\[
e^{-\wtd{\Phi}_i}=\frac{e^{-\wtd{\phi}_i}}{\left(1+e^{-\beta\wtd{\phi}_i}|\xi|^{2\beta}\right)^{\beta^{-1}}}.
\]
where $\xi=\xi_i$ is the coordinate of the second factor $\bC$. Then its curvature current $\wtd{\Omega}_i=\sddb \wtd{\Phi}_i$ is a (1,1)-current on $W_i\times \bC$ and can be calculated as in \eqref{curvproj}:
\[
\wtd{\Omega}_i=\frac{\wtd{\omega}}{1+e^{-\beta\wtd{\phi}_i}|\xi|^{2\beta}}+\frac{\beta e^{-\beta\wtd{\phi}_i}}{(1+e^{-\beta\wtd{\phi}_i}|\xi|^{2\beta})^2}\frac{(d\xi-\xi\partial\wtd{\phi}_i)\wedge\overline{(d\xi-\xi\partial\wtd{\phi}_i)}}{|\xi|^{2(1-\beta)}}
\]
$\wtd{\Omega}_i$ is smooth away from $W_i\times\{0\}$ and has conical singularities of angle $2\pi \beta$ along $W_i\times\{0\}$. 
$e^{-\wtd{\Phi}_i}$ descends to become a Hermitian metric on the $\bQ$-line bundle $\bar{f}^*H_{\orb}|_{\bar{f}^{-1}U_i}$ whose Chern curvature $\Omega_i$ is induced from $\wtd{\Omega}_i$ under the quotient $W_i\times\bC\rightarrow W_i\times\bC/\mu_m=\bar{f}^{-1}(U_i)$. Using the transition functions of the $\bQ$-line bundle $\bar{f}^*H_{\orb}$, it's easy to verify that $\Omega_i=\Omega_j$ on any intersection $U_i\cap U_j$ so that there is a globally defined (1,1)-current $\Omega$ on $\ocC^{\circ}$. 
The collection $e^{-\Phi}:=\{e^{-\Phi_i}\}$ defines a Hermitian metric on $\bar{f}^*H_\orb$ over $\ocC^\circ$ whose curvature is the globally defined (1,1)-current $\Omega$. 

Now by taking into account of the formula \eqref{ccprojcn}, we can carry out the same calculation in the proof of Proposition \ref{cKEproj} for the case $\Delta=0$ using the above equivariant charts. So we see that $e^{-\Phi}$ satisfies the conical K\"{a}hler-Einstein equation on $\ocC^{\circ}$:
\begin{equation}\label{orbKE}
(\sddb \Phi)^n=\frac{e^{-\frac{r(n+1)}{n}\Phi}}{|s_{V_\infty}|^{2(1-\beta)}},
\end{equation}
where $s_{V_\infty}$ is $\bQ$-section of the $\bQ$-line bundle $\cO_{\ocC}(V_\infty)$.
Because $\beta=r/n$ and, by Lemma \ref{lemKol}, $-K_\ocC=(r+1)V_\infty=\frac{r(n+1)}{n}V_\infty+(1-\frac{r}{n})V_\infty$, both sides of \eqref{orbKE} can be considered as singular Hermitian metrics on $-K_\ocC$. Arguing as before, \eqref{orbKE} actually holds on the whole $\ocC$ because both sides do not charge the vertex $o$. So we see that $e^{-\Phi}$ is indeed a cKE potential on $(\ocC_{\orb}, (1-\beta)V_\infty; \cO_{\ocC}(V_\infty))$.

\end{proof}

In the above theorem, we assumed that $(V, \Delta)$ is a smooth orbifold and the metrics under consideration are all orbifold smooth. However one can show that the smoothness assumptions are not necessary and the conclusion still holds for locally bounded orbifold metrics in the sense of pluripotential theory. So we can get a more general result as follows. Since we don't need this more general version in this paper, we leave the proof to the reader.
\begin{prop}
Let $f: Y^n\rightarrow (V^{n-1}, \Delta=\sum_{i} (1-\frac{1}{m_i})D_i)$ be a Seifert $\bC^*$-bundle.
Assume $c_1(Y/V)$ is positive and let $\ocC_{\orb}$ denote the associated orbifold projective cone. Assume that
$E$ is an orbifold $\bQ$-Cartier divisor such that the following two conditions are satisfied:
\begin{enumerate}
\item
The orbifold log canonical class $K_V+\Delta+E$ is anti ample and there is an orbifold conical K\"{a}hler-Einstein metric on $(V, \Delta; E)$;
\item
The Chern class $c_1(Y/V)=-r^{-1} (K_V+\Delta+E)$ for $0<r\le n$;
\item The following equality holds as an identity of effective $\bQ$-divisors:
\[
K_{\ocC^\circ}+V_\infty+\cE=\bar{f}^*(K_V+\Delta+E).
\]

\end{enumerate} 

Then there exists a conical K\"{a}hler-Einstein metric on $(\ocC_{\orb}, (1-\beta)V_\infty+\cE; \cO_{\ocC_\orb}(V_\infty))$ with $\beta=\frac{r}{n}$.
\end{prop}

\subsubsection{Appendix: Seifert \texorpdfstring{$\bC^*$}{C*}-bundles and orbifolds}\label{AppSeif}
For the reader's convenience, we recall the definition of Seifert $\bC^*$-bundle and their associated orbifolds after Koll\'{a}r. The following materials can be found in \cite{Kol04a, Kol04b}). 
\begin{defn}
Let $V$ be a normal complex space. A Seifert $\bC^*$-bundle over $V$ is a normal complex space $Y$ together with a morphism $f: Y\rightarrow V$ and a $\bC^*$-action on $Y$ such that the following two conditions are satisfied:
\begin{enumerate}
\item $f$ is Stein and $\bC^*$-equivariant (with respect to the trivial action on $V$);
\item For every $p\in V$, the $\bC^*$-action on the reduced fiber $Y_p:={\rm red} f^{-1}(p)$, $\bC^*\times Y_p\rightarrow Y_p$ is $\bC^*$-equivariantly biholomorphic to the natural $\bC^*$-action on 
$\bC^*/\mu_m$ for some $m=m(p, Y/V)$, where $\mu_m\subset\bC^*$ denotes the group of $m$th roots of unity. 
\end{enumerate} 
The number $m(p, Y/V)$ is called the multiplicity of the Seifert fiber over $p$.
\end{defn}
One can always assume that the $\bC^*$-action is effective, that is, $m(p, Y/V)=1$ for general $p\in V$.
Let $f: Y\rightarrow V$ be a Seifert $\bC^*$-bundle. The set of points $\{p\in V; m(p, Y/V)>1\}$ is a closed analytic subset of $V$. It can be decomposed as the union of Weil divisors
$\cup D_i$ and of a subset of codimension at least 2 contained in ${\rm Sing}(V)$. The multiplicity $m(p, Y/V)$ is constant on a dense open subset of each $D_i$. This common value is called the 
multiplicity of the Seifert $\bC^*$-bundle over $D_i$, and is denoted by $m_i=m(D_i)$. The $\bQ$-divisor $\Delta:=\sum_{i} (1-\frac{1}{m_i})D_i$ is called the {\it branch divisor} of $f: Y\rightarrow V$. We will often 
write $f: Y\rightarrow (V, \Delta)$ to indicate the branch divisor. 

Koll\'{a}r \cite{Kol04a} classified general Seifert $\bC^*$-bundles over $V$, generalizing the previous special cases of Dolgachev-Pinkham-Demazure.
\begin{defn}[see \cite{Kol04a}]
Let $V$ be a normal variety, $L$ a rank 1 reflexive sheaf on $V$, $D_i$ distinct irreducible divisors and $0<s_i<1$ rational numbers. Define:
\begin{eqnarray*}
S(L, \sum_i s_i D_i)&:=&\sum_{j\in \bZ}L^{[j]}\left(\sum_i \lfloor j s_i\rfloor D_i \right)\\
Y(L, \sum_i s_i D_i)&:=&{\rm Spec}_V S(L, \sum_i s_i D_i).
\end{eqnarray*}
where $L^{[k]}=(L^{\otimes k})^{**}$ for the rank 1 reflexive sheaf $L$ (see \cite{Kol04a}).
\end{defn}
There is a natural $\bC^*$-action on $S(L, \sum_i s_i D_i)$ where $L^{[j]}\left(\sum_i s_i D_i\right)$ is the $\lambda^j$ eigensubsheaf. This induces a $\bC^*$-action on 
$Y(L, \sum_i s_i D_i)$. 
\begin{thm}[Koll\'{a}r]
Let $V$ be a normal variety. 
\begin{enumerate}
\item
Every Seifert $\bC^*$-bundle $f: Y\rightarrow V$ can be uniquely written as $Y\cong Y(L, \sum_i s_i D_i)$ for some $L$ and $\sum_i s_i D_i$. 
\item
$f: Y(L, \sum_i s_i D_i)\rightarrow V$ is a Seifert $\bC^*$-bundle if and only if $L^{[M]}(\sum_i (M s_i)D_i)$ is locally free for some $M>0$ and $M s_i$
is an integer for every $i$.
\end{enumerate}
\end{thm}

\begin{defn}
The {\it Chern class} of $f: Y\rightarrow (V, \Delta)$ is defined as 
\[
c_1(Y/V):=L+\sum_i s_i [D_i]\in H^2(V, \bQ).
\]
\end{defn}

Seifert $\bC^*$-bundles can be compactified by adding the zero and infinite sections. Their union gives a proper morphism $\bar{f}: \bar{Y}\rightarrow V$ which is a $\bP^1$-bundle over
the set where $m(p)=1$. We have the zero section $V_0\subset \bar{Y}$ and the infinity section $V_\infty\subset \bar{Y}$. The zero section $V_0\subset \bar{Y}$ can be contracted to a point
$\{o\}$ if and only if $c_1(Y/V)$ is positive. In this case, we will denote by $\ocC:=\ocC_{\orb}=\ocC(Y/V)$ the projective variety obtained from $\bar{Y}$ by contracting $V_0$ and call it the orbifold projective cone. We will also denote by 
$\ocC^{\circ}$ the punctured projective cone $\ocC\setminus\{o\}$ and still by $\bar{f}: \ocC^{\circ}\rightarrow V$ the restriction of $\bar{f}$.

\begin{lem}[{\cite[40-42]{Kol04a}}]\label{lemKol}
Suppose $c_1(Y/V)$ is positive and denote by $\ocC$ the orbifold projective cone over $V$.
Then $K_\ocC$ is $\bQ$-Cartier if and only if $c_1(Y/V)=-r^{-1}(K_V+\Delta)$ with $r\in \bQ_{>0}$. In this case $\ocC$ has klt singularities and as $\bQ$-divisors on $\ocC$, 
\begin{equation}\label{ccprojcn}
K_{\ocC}\sim_{\bQ} -(1+r)V_\infty.
\end{equation}

\end{lem}
\begin{proof}
As pointed out in \cite[42]{Kol04a}, the first statement follows from the formula for canonical divisor of $Y$: 
$K_Y=f^*(K_V+\Delta)$ and the description of the class group of Seifert bundles by Flenner-Zaidenberg.
If $t$ is a coordinate on each $\bC^*$-fibre, then $dt/t$ is a well defined 1-form on $\ocC^\circ$ with pole order 1 along $V_\infty$. Using this we see that
$K_{\ocC^\circ}(V_\infty)=\bar{f}^*(K_V+\Delta)$. So \eqref{ccprojcn} follows from the fact that, as $\bQ$-divisors, $V_\infty\sim_\bQ \bar{f}^* c_1(Y/V)$. 
$\ocC$ has quotient singularities and hence klt singularities near $V_\infty$. On the other hand, using the same argument as in \cite[Proof of Lemma 3.1]{Kol13}, we know that $\ocC$ has klt singularities at the vertex. 

\end{proof}

\begin{defn}
An orbifold is a normal, compact complex space $V$ covered by local charts given as quotients of smooth coordinate charts. In other words, $V$
can be covered by open charts $V=\cup_i U_i$, and for each $U_i$ there is a smooth complex space $W_i$ and a finite group $G_i$ acting on $W_i$ such that 
$U_i$ is biholomorphic to the quotient space $W_i/G_i$. The quotient maps will be denoted by $\pi_i: W_i\rightarrow U_i$.

One needs to assume the compatibility condition between the charts: there are global divisors $D_j\subset X$ and ramification indices $m_j$ such that
$D_{ij}=U_i\cap D_j$ and $m_{ij}=m_j$ (after suitable re-indexing). The branch divisor of the orbifold is defined as $\Delta:=\sum_i (1-\frac{1}{m_j})D_j$ and the orbifold canonical
class is defined as $K_V+\sum_i (1-\frac{1}{m_i})D_i$. This is the negative of the orbifold Chern class:
\[
c_1^{\orb}(V, \Delta):=c_1(X)-\sum_i (1-\frac{1}{m_i})[D_i]\in H^2(V, \bQ).
\]

\end{defn}

Let $f: Y\rightarrow (V, \Delta)$ be a Seifert $\bC^*$-bundle with $Y$ smooth. For $p\in V$ pick any 
$y\in f^{-1}(p)$ and a $\mu_m$-invariant smooth hypersurface $W_p\subset Y$ transversal to ${\rm red} f^{-1}(p)$ for $m=m(p, Y/V)$. Then $\{\pi_p: W_p\rightarrow U_p:=W_p/\mu_m\}$ gives
an orbifold structure on $V$. The orbifold branch divisor coincides with the branch divisor of the Seifert bundle. The orbifolds coming from a smooth Seifert bundle have additional property that each 
$U_p$ is a quotient by a cyclic group $\mu_m$. Such an orbifold is called {\it locally cyclic}.

It's useful to use the local structure of near $V_\infty$. We can choose an open set $U=W/\mu_m$ such that 
$\bar{f}^{-1}(U)\cong \bC\times \bC^{n-1}/\mu_m(p, a_1, \cdots, a_{n-1})$ such that ${\rm gcd}(a_1, \dots, a_{n-1}, m)=1$. Following Koll\'{a}r, we define the integers:
\[
c_i:={\rm gcd}(a_1, \dots, \widehat{a_i}, \dots, a_{n-1}, m), \quad d_i:=a_i c_i/C, \quad C:=\prod_i c_i.
\]
Then $c_i$ are pairwise relatively prime and $C/c_i$ divides $a_i$. Moreover as a variety:
\[
V=A^{n-1}_{\bf z}/\mu_m(a_1, \dots, a_{n-1})\cong A^n_{\bf x}/\mu_{m/C}(d_1, \dots, d_n).
\]
Denote $M=m/C$ and define the $\bQ$-divisors:
\[
D_i: (x_i=0)/\mu_M(d_1, \dots, \widehat{d_i}, \dots, d_{n-1})\subset V.
\]
and let
\[
\cO_V(j):=\epsilon^j\text{-eigenspace of } \bC[x_1, \dots, x_{n-1}] 
\]
as an $\cO_X$-module. Then $\cO_V(D_j)\cong \cO_V(d_j)$. Moreover, since $\mu_M$ acts on $A^n_{\bf x}$ without pseudo-reflections, $K_{A^n_{\bf x}}=\pi^* K_V$. So $K_{V}=\cO_V(-\sum d_i)=\cO_V(-\sum D_i)$ and
\[
K_{\orb}=K_V+\sum_i (1-\frac{1}{c_i})D_i=-\sum_i \frac{D_i}{c_i}.
\]
so that
\[
C\cdot K_{\orb}=-\sum_i \frac{C}{c_i}D_i=-\sum_i \frac{a_i }{d_i}D_i=\cO_V(-\sum_i a_i).
\]
One can write $p$ uniquely as $p\equiv l C+\sum a_i b_i$ mod $M$ where $0\le b_i<c_i$ for every $i$ and 
\[
C\cdot c_1(Y/X)=\cO_V\left(Cl+\sum_i \frac{C}{c_i} b_i d_i\right)=\cO_V(Cl+\sum a_i b_i)=\cO_V(p).
\] 
If $c_1(Y/X)=-r^{-1}K_\orb$, then $p=r^{-1}\sum_i a_i$ and locally
\[
\bar{f}^{-1}(U)=U\times (r^{-1} K_U^{-1}) /\mu_m. 
\]
\begin{defn}
An orbifold Hermitian metric $g$ on the orbifold $(V, \Delta)$ is a Hermitian metric $g$ on $V\setminus ({\rm Sing V}\cup {\rm Supp} \Delta)$ such that for every chart $\pi_i: W_i\rightarrow U_i$ the pull back $\pi_i^* g$ extends to a Hermitian metric on $W_i$. Similarly, geometric objects on complex manifolds can be extended in a straight forward manner to the orbifold setting by working with local uniformizing charts. For example one can talk about curvature, K\"{a}hler metrics, K\"{a}hler-Einstein metrics on orbifolds. 
\end{defn}

\section{Proof of main theorems}
\subsection{Log version of Fujita's result}\label{seclogFujita}

Following \cite{Fuj15}, we will first translate Berman's expansion of log-Ding-energy into an algebraic version. 
\begin{prop}[{\cite[Proposition 3.5]{Fuj15}}]\label{logBermAlg}
Let $(X, D)$ be a $n$-dimensional log-$\bQ$-Fano pair which is log-Ding-semistable. Let $\delta$ be a positive rational number such that
$-\delta^{-1}(K_X+D)$ is Cartier. Let $I_M\subset\cdots\subset I_1\subset \cO_X$ be a sequence of coherent ideal sheaves and let 
$\cI:=I_M+I_{M-1}t^1+\cdots+I_1 t^{M-1}+(t^M)\subset \cO_{X\times\bC^1}$. 
Let $\Pi: \cX\rightarrow X\times \bC^1$ be the blowup along
$\cI$. Denote $\cD:=\Pi^*(D\times\bC^1)$.
Let $E\subset \cX$ be the Cartier divisor defined by $\cO_X(-E)=\cI\cdot \cO_\cX$, and let
\[
\cL:=\Pi^*\cO_{X\times\bC^1}(-\delta^{-1}\left(K_{X\times\bC^1/\bC^1}+
D\times\bC^1)\right)\otimes\cO_{\cX}(-E).
\] 
Assume that $\cL$ is semiample over $\bC^1$. Then $(\cX,\cD;\cL)$ is naturally
seen as a (possibly non-normal) semi test configuration of $(X,D;-\delta^{-1}
(K_X+D))$.
Under these conditions, $((X\times\bC^1, D\times\bC^1); \;
\cI^\delta \cdot (t)^{d})$ 
must be sub log canonical, where
\[
d:=1+\frac{\delta^{n+1}(\bar{\cL}^{n+1})}{(n+1)((-K_X-D)^n)}.
\]
Moreover, we have the equality:
\[
(\bar{\cL}^{n+1})=-\lim_{k\rightarrow+\infty}\frac{\dim\left(\frac{H^0(X\times\bC^1, \cO_{X\times\bC^1}(-k\delta^{-1} (K_{X\times\bC^1/\bC^1}+D\times\bC^1))}{H^0(X\times\bC^1, \cO_{X\times\bC^1}(-k\delta^{-1}(K_{X\times\bC^1/\bC^1}+D\times\bC^1))\cdot \cI^k)}\right)}{k^{n+1}/(n+1)!}.
\]
\end{prop}
\begin{proof}
The proof is the same as Fujita's proof. Let $\nu: \cX^\nu\rightarrow \cX$  
be the normalization. Denote $\cD^\nu:=\nu^*\cD$. Then $(\cX^\nu,\cD^\nu;
\nu^*\cL)/\bC^1$ is a normal semi test configuration of $(X,D;-\delta^{-1}
(K_X+D))$.
By the log-Ding-semistability of $(X,D)$, we can deduce that 
$\lDing(\cX^\nu,\cD^\nu; \nu^*\cL)\ge 0$. 
Then we notice  the identity:
\begin{eqnarray*}
&&\cO_{\bar{\cX}^\nu}\left(\delta^{-1}\left(K_{\bar{\cX}^\nu/\bP^1}
+\cD^\nu-\Delta_{(\cX^\nu,\cD^\nu; \nu^*\cL)}\right)\right)\cong \nu^*\bar{\cL}^{-1}\\
&=& \nu^*\cO_{\bar{\cX}}\left(\delta^{-1}\left(\Pi^*(K_{X\times\bP^1/\bP^1}+D\times\bP^1)+\delta E\right)\right).
\end{eqnarray*}
So $(\cX^\nu, \cD^\nu-\Delta_{(\cX^\nu,\cD^\nu; \nu^*\cL)}+c\cX^\nu_0)$ is sub log canonical if and only if $((X\times\bC^1, D\times\bC^1);\; \cI^{\delta}\cdot (t)^c)$ is sub log canonical. Thus we have
the identity:
\[
\lct(\cX^\nu, \cD^\nu-\Delta_{(\cX^\nu,\cD^\nu; \nu^*\cL)}; \cX^\nu_0)=
\lct\left((X\times\bC^1, (D\times\bC^1)\cdot \cI^\delta); (t)\right).
\]
The rest of the argument is the same as in \cite{Fuj15} and \cite{Oda}. 
\end{proof}

Assume $(X, D)$ is a log pair such that $L=-\delta^{-1}(K_X+D)$ is an ample Cartier divisor. Denote by $S=\bigoplus_{m=0}^{+\infty} H^0(X, mL)$ the graded $\bC$-algebra of sections of $(X, L)$. Assume $\cF$ is a good filtration of the graded ring $S$ as in Definition \ref{defil}. Following Fujita, define the ideal sheaf:
\[
\cI^{\cF}_{(m,x)}={\rm Im}\left(\cF^x S_m\otimes L^{-m}\rightarrow \cO_{X}\right),
\]
and the ideal sheaf on $X\times\bC$:
\[
\cI_m=\cI^{\cF}_{(m,m e_+)}+\cI^{\cF}_{(m, m e_+-1)}t^1+\cdots+
\cI^{\cF}_{(m,m e_-+1)}t^{m(e_+-e_-)-1}+\left(t^{m(e_+-e_-)}\right).
\]

By \cite[4.3]{Fuj15}, there exists $m_1\in\bZ_{>0}$ such that $\{\cI_m\}_{m\geq m_1}$
is a graded family of coherent ideal sheaves on $X\times \bC^1$. 
For any $m\geq m_1$, let
\begin{itemize}
 \item $\Pi_m:\cX_m\to X\times\bC^1$ be the blow up along $\cI_m$,
 \item $\cD_m:=\Pi_m^*(D\times\bC^1)$,
 \item $E_m\subset\cX_m$ be the Cartier divisor defined by $\cO_{\cX_m}(-E_m)
 =\cI_m\cdot\cO_{\cX_m}$, and
 \item $\cL_m:=\Pi_m^*\cO_{X\times\bC^1}(-m\delta^{-1}\left(K_{X\times\bC^1/\bC^1}+
D\times\bC^1)\right)\otimes\cO_{\cX_m}(-E_m)$.
\end{itemize}
\begin{prop}[{\cite[4.6]{Fuj15}}]
 The line bundle $\cL_m$ is semiample over $\bC^1$. Thus $(\cX_m,\cD_m;\cL_m)/\bC^1$
 is a semi test configuration of $(X,D;L^m)$. 
\end{prop}

Thus, by Proposition \ref{logBermAlg}, $((X\times\bC^1, D\times\bC^1); 
\cI_m^{\delta/m}\cdot(t)^{d_m})$ is sub log canonical, where
\[
 d_m:=1+\frac{\delta^{n+1}(\bar{\cL}_m^{n+1})}{(n+1)m^{n+1}((-K_X-D)^n)}.
\]

\begin{prop}[log-Fujita, cf. \cite{Fuj15}]\label{logFujita}
Assume the log pair $(X, D)$ satisfies the following properties:
\begin{itemize}
\item $K_X+D$ is $\bQ$-Cartier; 
\item $\left(X, D\right)$ is a klt log-Fano pair;
\item $(X, D)$ is log-Ding-semistable; 
\item $L=-\delta^{-1} (K_X+D)$ is an ample Cartier divisor;
\item $\cF$ is a good graded filtration  of $S=\bigoplus_{k=0}^{+\infty} H^0(X, k L)$ (see Definition \ref{defil}).
\end{itemize}
Then $((X\times\bC^1, D\times\bC^1); \cI^{\delta}_\bullet\cdot(t)^{d_\infty})$ is sub log canonical,
where
\begin{align*}
 \cI_m= &\cI^{\cF}_{(m,m e_+)}+\cI^{\cF}_{(m, m e_+-1)}t^1+\cdots+
\cI^{\cF}_{(m,m e_-+1)}t^{m(e_+-e_-)-1}+\left(t^{m(e_+-e_-)}\right),\\
d_\infty = & 1-\delta(e_+ - e_-)+\frac{\delta}{(L^n)}
\int_{e_-}^{e_+}\vol\left(\overline{\cF}S^{(t)}\right)dt.
\end{align*}

\end{prop}
\begin{proof}
The proof is the same as in \cite{Fuj15}.
Indeed by \cite[4.7]{Fuj15} we have that
$\lim_{m\to\infty}d_m=d_\infty$. Then the proposition follows by 
\cite[2.5(1)]{Fuj15}.
\end{proof}

\subsection{Application of log-Fujita to the filtrations of valuations}\label{secvalfil}
Let $o\in X$ be a closed point. Let $v\in \Val_{X, o}$ be a real valuation centered at $o$. 
We will apply the above log version of Fujita's result to the filtration
associated to $v$ on $X$. 
Consider the following graded filtration of $S$. For an $x\in \bR$ define:
\[
\cF^x S_m=H^0(X, L^m\otimes \fa_x),
\]
where $\fa_x:=\fa_x(v)=\{f\in \cO_X: v(f)\ge x\}$.
\begin{lem}\label{valfil}
$\cF$ is a decreasing,
left-continuous, multiplicative and saturated filtration of $S$. Moreover, if
$A_{(X,D)}(v)<+\infty$ then $\cF$ is also linearly bounded.
\end{lem}
\begin{proof}
It's clear the $\cF$ is decreasing, left-continuous and multiplicative. To prove that $\cF$ is saturated notice that the homomorphism
\[
\cF^x S_m\otimes_{\bC} L^{\otimes(-m)}\twoheadrightarrow I^{\cF}_{(m,x)}
\]
induces the inclusion $I^{\cF}_{(m,x)}\subset\fa_x$ for any $x\in\bR$. Thus
$\overline{\cF}^x S_m= H^0(X,L^{\otimes m}\cdot I^{\cF}_{(m,x)})\subset \cF^x S_m$.

If $A_{(X,D)}(v)<+\infty$, then by a log version of \cite[Proposition 1.2]{Li15a} 
there exists a constant $c=c(X,o)$ such that $v\le c A_{(X,D)}(v)\ord_o$. So it's easy to see 
that 
\[
e(S_\bullet, \cF)\le c A_{(X,D)}(v)\cdot e_{\max}(S_\bullet, \cF_{\ord_o})<+\infty.
\]
The second inequality follows from the work in \cite{BKMS14}.

\end{proof}

It is clear that $\cF^x S_m= S_m$ for $x\leq 0$. Hence we may choose $e_-=0$.
The graded family of ideal sheaves $\cI_\bullet$ on $X\times\bC^1$ becomes:
\[
\cI_m=\cI^{\cF}_{(m,m e_+)}+\cI^{\cF}_{(m, m e_+-1)}t^1+\cdots+\cI^{\cF}_{(m,1)}t^{m e_+-1}+\left(t^{me_+}\right).
\]
We also have that
\[
 d_\infty=1-\delta e_+ + \frac{\delta}{(L^n)}\int_{0}^{\infty}\vol\left(\overline{\cF}S^{(t)}\right)dt.
\]

The valuation $v$ extends to a $\bC^*$-invariant valuation $\bar{v}$ on
$\bC(X\times\bC)=\bC(X)(t)$, such that for any 
$f=\sum_{k}f_k t^k$ we have:
\[
\bar{v}(f)=\min_{k}\{v(f_k)+k\}.
\]

\begin{prop}\label{logFujval}
If $(X, D)$ is log-Ding-semistable and $L=-\delta^{-1}(K_X+D)$ then
\begin{equation}\label{Avint}
A_{(X,D)}(v)-\frac{\delta}{(L^{n})}\int^{+\infty}_0 \vol\left(\cF S^{(t)}\right)dt\ge 0.
\end{equation}

\end{prop}
\begin{proof}
We may assume that $A_{(X,D)}(v)<+\infty$, since otherwise the inequality holds
automatically. Hence Lemma \ref{valfil} implies that $\cF$ is good and saturated.

By Proposition \ref{logFujita}, we know that $((X\times\bC^1, D\times\bC^1); \;
\cI_\bullet^\delta\cdot (t)^{d_\infty})$ is sub log canonical.
Since $(X, D)\times\bC^1:=(X\times\bC^1, D\times\bC^1)$ has klt singularities,
for any $0<\epsilon\ll 1$ there exists $m=m(\epsilon)$ such that
\[
 \cO_{X\times\bC^1}\subset\cJ\left((X\times\bC^1, D\times\bC^1);
\cI_m^{(1-\epsilon)\delta/m}\cdot (t)^{(1-\epsilon)d_\infty}\right).
\]
By \cite[1.2]{BFFU13} we know that the following inequality holds for any real
valuation $u$ on $X\times\bC^1$:
\begin{equation}
 A_{(X, D)\times \bC^1}(u)>\frac{(1-\epsilon)\delta}{m}u(\cI_m)+(1-\epsilon)d_\infty u(t).
\end{equation}
Let us choose a sequence of quasi-monomial real valuations $\{v_n\}$ on $X$
such that $v_n\to v$ and $A_{(X,D)}(v_n)\to A_{(X,D)}(v)$ when $n\to\infty$. It is easy
to see that $\{\bar{v}_n\}$ is a sequence of quasi-monimial valuations on
$X\times\bC^1$ satisfying $\bar{v}_n(t)=1$ and 
\[
 A_{(X,D)\times\bC^1}(\bar{v}_n)=A_{(X,D)}(v_n)+1.
\]
Hence we have
\begin{align*}
A_{(X,D)}(v)+1  &= \lim_{n\to\infty} A_{(X,D)}(v_n)+1=\lim_{n\to\infty} A_{(X,D)\times\bC^1}(\bar{v}_n)\\
& \geq \frac{(1-\epsilon)\delta}{m}\lim_{n\to\infty}\bar{v}_n(\cI_m) + (1-\epsilon)d_\infty.
\end{align*}
From the definition of $\cF^x S_m$ we get:
\[
v\left(\cI^{\cF}_{(m,x)}\right)\ge x.
\]
Therefore,
\begin{align*}
 \lim_{n\to\infty}\bar{v}_n(\cI_m) & =\lim_{n\to\infty} \min_{0\leq j\leq re_+}\left\{v_n\left(\cI^{\cF}_{(m,j)}\right)+ me_+ -j\right\}\\
 & =\min_{0\leq j\leq re_+}\left\{\lim_{n\to\infty}v_n\left(\cI^{\cF}_{(m,j)}\right)+ me_+ -j\right\}\\
 & =\min_{0\leq j\leq re_+}\left\{v\left(\cI^{\cF}_{(m,j)}\right)+me_+ - j\right\}\\
 & \geq me_+
\end{align*}
Hence when $\epsilon\to 0+$ we get:
\begin{equation}
 A_{(X,D)}(v)+1\geq \delta e_+ + d_\infty.
\end{equation}
Therefore,
\begin{align*}
 A_{(X,D)}(v)
 & \geq -1 + \delta e_+ + d_\infty\\
 & = \frac{\delta}{(L^{n})}\int^{+\infty}_0 \vol\left(\cF S^{(t)}\right)dt.
\end{align*}
Hence we get the desired inequality.
\end{proof}

\begin{prop}\label{proplb}
If $(X, D)$ is log-Ding-semistable, then for any real valuation $v$ centered at any closed point of $X$, we have the estimate:
\begin{equation}\label{ineqhvol}
\hvol_{(X,D)}(v)\ge  \left(\frac{n}{n+1}\right)^n (-K_X-D)^n.
\end{equation}
\end{prop}

\begin{proof}
We prove this estimate using the method in \cite{Fuj15, Liu16}.
We may assume $A_{(X,D)}(v)<+\infty$, since otherwise the inequality holds automatically.
Since $\cF^{mx} S_m=H^0(X,L^{\otimes m}\cdot\fa_{mx})$, we have the exact sequence
\[
 0\to \cF^{mx}S_m\to H^0(X,L^{\otimes m})\to H^0(X,L^{\otimes m}\otimes(\cO_X/\fa_{mx})).
\]
Hence we have
\[
 \dim \cF^{mx}S_m\geq h^0(X,L^{\otimes m})-\ell(\cO_X/\fa_{mx}).
\]
Dividing by $m^n/n!$ and taking limits as $m\rightarrow+\infty$, we get 
\begin{equation}\label{liuest}
\vol\left(\cF S^{(x)}\right)\ge (L^n)-\vol(v)x^n.
\end{equation}
So we have the estimate:
\begin{align*}
\int^{+\infty}_0\vol\left(\cF S^{(x)}\right)dx &\ge \int_0^{\sqrt[n]{(L^n)/\vol(v)}}
\left((L^n)-\vol(v)x^n\right)dx\\
&=\frac{n}{n+1}(L^n)\cdot \sqrt[n]{(L^n)/\vol(v)}.
\end{align*}
Applying \eqref{logFujval} we get the inequality:
\begin{eqnarray*}
A_{(X,D)}(v)&\ge&\frac{\delta}{(L^{n})}\int^{+\infty}_0\vol\left(\cF S^{(t)}\right)dt\\
&\ge& \frac{\delta}{(L^n)}\frac{n}{n+1} (L^{n})\sqrt[n]{(L^n)/\vol(v)}=\frac{n}{n+1}\sqrt[n]{(\delta L)^n/\vol(v)}.
\end{eqnarray*}
Since $\delta L=-K_X-D$, this is exactly equivalent to
\[
\hvol_{(X,D)}(v)=A_{(X,D)}(v)^n\vol(v)\ge \left(\frac{n}{n+1}\right)^n (-K_X-D)^n.
\]
\end{proof}

\begin{proof}[Proof of Theorem \ref{thmlog}]
By Section \ref{seclogK}, we apply the above proposition to the case $X=\ocC$ and $D=(1-\beta)V_\infty+\cE$. Recall that by \eqref{Lcone} we have: 
\begin{eqnarray*}
-(K_X+D)&=&-(K_X+(1-\beta)V_\infty+\cE)=r\frac{n+1}{n}V_\infty.
\end{eqnarray*}
So we get the intersection number:
\begin{equation}\label{logint}
(-K_X-D)^n=r^n \left(\frac{n+1}{n}\right)^n H^{n-1}.
\end{equation}
Because $v$ is centered at $o\in \cC$, $v(V_\infty)=0$ and hence $A_{(X, D)}(v)=A_{(\ocC, \cE)}(v)-(1-\beta)v(V_\infty)=A_{(\cC, \cE)}(v)$. On the other hand, $A_{(\cC, \cE)}(v_0)=r$ (see \cite[Section 3.1]{Kol13}) and $\vol(v_0)=(H^{n-1})$. So $\hvol_{(\cC, \cE)}(v_0)=r^n H^{n-1}=\hvol(v_0)$. Substituting \eqref{logint} into \eqref{ineqhvol}, we get:
\begin{equation}\label{eqthmlog}
\hvol_{(\cC, \cE)}(v)=A_{(\cC, D)}^n\cdot \vol(v)\ge r^n H^{n-1}=\hvol(v_0)=\hvol_{(\cC, \cE)}(v_0).
\end{equation}

\end{proof}

\begin{proof}[Proof of Theorem \ref{thmsemi}]
Choose $F\in |-mK_V|$ for $m\gg 1$ such that $F$ is an irreducible prime divisor satisfying:
\begin{enumerate}
\item $F$ does not contain the center of $v$;
\item $(V, (1-\beta) F)$ is klt for any $0<\beta<1$.
\end{enumerate}
Consider the effective divisor $E_\kappa=(1-\kappa) F/m$
with $0<\kappa\le 1$. Then $(V, E_\kappa)$ is a log-Fano pair:
\[
-K_V-E_\kappa=-K_V-(1-\kappa)(-K_V)=\kappa (-K_V).
\]
Using Tian's $\alpha$-invariant as in \cite{Bm13, LS12, SW12} (see also \cite{JMR15, BHJ15}), we know that the log-Ding-energy is proper and there exists a conical K\"{a}hler-Einstein potential on $(V, E_\kappa; -K_V)$ for $0<\kappa\ll 1$. 

Now if $V$ specially degenerates to a K\"{a}hler-Einstein variety then its Ding-energy is uniformly bounded from below by \cite[Theorem 4]{Li13} (and hence $V$ is K-semistable).
Notice that the assumption in \cite[Theorem 4]{Li13} is that the generic fiber $V$ is smooth. However the proof works for general $\bQ$-Fano variety $V$. Indeed the key calculation showing the continuity of log-Ding-energy there is generalized later in \cite[Appendix I]{LWX15}.
So by using interpolation argument as in \cite{LS12}, we know that there is a conical K\"{a}hler-Einstein potential on $(V, E_\kappa; -K_V)$ for any $\kappa\in (0,1)$. 

Because $F$ does not contain the center
of $v$, $v(\cI_{\cE_\kappa})=0$ and hence $A_{(X,\cE_\kappa)}(v)=A_X(v)$. On the other hand, $H=-r^{-1}K_V=-(\kappa r)^{-1}(K_V+E_\kappa)$. So by Theorem \ref{thmlog} (see \eqref{eqthmlog}), 
\[
\hvol(v)=\hvol_{(X, \cE_\kappa)}(v)\ge (\kappa r)^n H^{n-1}=\kappa^n \hvol(v_0). 
\]
The conclusion follows by letting $\kappa\rightarrow 1$.

If $(\cC, (1-\beta)V_\infty)$ specially degenerates to a conical K\"{a}hler-Einstein pair, then its log-Ding-energy is uniformly bounded from below by \cite[Theorem 4]{Li13} as above. So the pair is log-Ding-semistable by the proof of Berman's result in Proposition \ref{logBerm}. Hence we can directly apply log-Fujita in Proposition \ref{logFujita} as in the proof Theorem \ref{thmlog}. 
\end{proof}

\begin{proof}[Proof of Theorem \ref{thmorb}]
The first statement is just Proposition \ref{propcKEorb}. For the second statement, we can use the same argument as in the proof of Theorem \ref{thmlog} and Theorem \ref{thmsemi} by applying Proposition \ref{proplb} to the case $X=\ocC_\orb$ and $D=(1-\beta)V_\infty$.
\end{proof}

\subsection{Examples}

\begin{exmp}
Consider the $n$-dimensional $A_{k-1}$ singularity. 
\[
A^n_{k-1}:=\{z_1^2+\cdots+z_n^2+z_{n+1}^{k}=0\}\subset\mathbb{C}^{n+1}. 
\]
$A^n_{k-1}$ is an orbifold affine cone with the orbifold base $(V, \Delta)$ given by the hypersurface in weighted projective space $\{Z_1^2+\cdots+Z_n^2+Z_{n+1}^k=0\}\subset\bP^n(k, \cdots, k, 2)$.
It's easy to see that the following
\begin{itemize}
\item If $k$ is odd, $(V, \Delta)=(\bP^{n-1}, (1-\frac{1}{k})Q'^{n-2})$ where $Q'^{n-1}=\{Z_1^2+\cdots+Z_n^2=0\}\subset\bP^{n-1}$;
\item If $k$ is even, $(V, \Delta)=(Q^{n-1}, (1-\frac{2}{k})Q^{n-2})$ where $Q^{n-1}=\{Z_1^2+\cdots+Z_{n+1}^2=0\}\subset\bP^{n+1}$ and $Q^{n-2}=Q^n\cap \{Z_{n+1}=0\}$.
\end{itemize}
The above two cases are related by using the 2-fold branched covering $\tau: Q^{n-1}\rightarrow (\bP^{n-1}, \frac{1}{2}Q'^{n-2})$ so that 
$K_{Q^{n-1}}=\tau^*\left(K_{\bP^{n-1}}+\frac{1}{2}Q'^{n-2}\right)$ and hence
$K_{Q^{n-1}}+(1-\frac{2}{k})Q^{n-2}=\tau^*(K_{\bP^{n-1}}+(1-\frac{1}{k})Q'^{n-2})$.
By \cite{GMSY07, LS12, Li13}, there is a conical K\"{a}hler-Einstein metric on $(V, \Delta)$ if and only if one of the following conditions hold:
\begin{enumerate}
\item $k=1$ or $2$, and $n$ is any positive integer;
\item $n=2$ and $k$ is any positive integer;
\item $n=3$ and $1\le k\le 3$.
\end{enumerate}
There is a natural $\bC^*$-action given by: $(z_1, \cdots, z_n, z_{n+1})\rightarrow (t^k z_1, \cdots, t^k z_n, t^2 z_{n+1})$. The associated valuation is denoted by $v_0$.
Notice that $(V, \Delta)=\left(A^n_{k-1}-\{0\}\right)/\bC^*$. By Theorem \ref{thmorb}, if any of the above conditions is satisfied, $\hvol(v)$ is globally minimized at the valuation $v_0$ associated to canonical $\bC^*$-action. 

Also by \cite{LS12} and \cite{Li13}, $(V, \Delta)$ is 
log-K-semistable but not log-K-polystable if and only if $(n,k)=(3, 4)$ or $(n, k)=(4, 3)$. Using the orbifold-version Theorem \ref{thmorb}, we can show that in these two cases $\hvol(v)$ is also globally minimized at $v_0$. We will prove this for the case $(n, k)=(3, 4)$ and the argument for the case $(n,k)=(4,3)$ is similar. When $(n,k)=(3,4)$, $(V, \Delta)=(\bQ^2, (1-\frac{2}{4})\bQ^1)\cong (\bP^1\times\bP^1, \frac{1}{2}E)$ where $E$ denotes the diagonal $\bP^1$. The compactified orbifold cone $\ocC_\orb$ is just the natural
compactification of $A^3_2$ inside $\bP^4(1,2,2,2,1)$. In other words, if the latter is given weighted homogeneous coordinates $[Z_0, Z_1, \cdots, Z_4]$ such that $z_i=Z_i/Z_0^2$ for $1\le i\le 3$ and $z_4=Z_4/Z_0$, then we have:
\[
\ocC_\orb=\{Z_1^2+Z_2^2+Z_3^2+Z^4_4=0\}\subset\bP^4(1,2,2,2,1).
\]

By the construction in \cite{LS12, Li13} (see also \cite[Section 5.3]{Li15b}), we know that $(V, \Delta)$ specially degenerates into $(\bP^2(1,1,2), \frac{1}{2}F)=:(\cV_0, \frac{1}{2}F)$, where $F=\{W_2=0\}$ if $[W_0, W_1, W_2]$ are weighted homogeneous coordinates of $\bP(1,1,2)$. This can be realized as a degeneration inside $\bP^3(2, 2, 2, 1)$ where the two pairs are realized as weighted projective hypersurfaces:
\begin{eqnarray*}
\left(\bP^1\times\bP^1, \frac{1}{2}E\right)&=&\{Z_1^2+Z_2^2+Z_3^2+Z_4^4=0\}\subset\bP^3(2,2,2,1);\\
\left(\bP^2(1,1,2), \frac{1}{2}F\right) &=& \{Z_1^2+Z_2^2+Z_3^2=0\}\subset\bP^3(2,2,2,1).
\end{eqnarray*}
Correspondingly, $\ocC_{\orb}$ degenerates to the weighted projective cone over $\bP^2(1,1,2)$ given by:
\[\cX_0:=\{Z_1^2+Z_2^2+Z_3^2=0\} \subset \bP^4(1,2,2,2,1).
\] 
Because $(\bP^2(1,1,2), \frac{1}{2}F)$ has a natural orbifold
K\"{a}hler-Einstein metric by pushing forward the Fubini-Study metric of $\bP^2$ under the 2-fold branched covering map $(\bP^2\rightarrow \bP^2(1,1,2), \frac{1}{2}F)$, by Proposition \ref{propcKEorb}, there is a conical K\"{a}hler-Einstein metric on the pair $(\cX_0, (1-\beta)(\cV_0)_\infty)$ where $(\cV_0)_\infty\cong \cV_0=\bP^2(1,1,2)$. To 
see the cone angle $\beta=r/n$, we notice that
\[
-K_{\cX_0}=4H=4\left(\{Z_0=0\}\cap \cX_0\right)=4(\cV_0)_\infty,
\]
where $H$ is the weighted hyperplane divisor of $\bP^{4}(1,2,2,2,1)$. So we get $r=3$ by comparing with \eqref{ccprojcn}. Alternatively, we can calculate the orbifold canonical class $K_{(\bP^2(1,1,2), \frac{1}{2}F)}=-4H'+H'=-3H'$ where $H'$ is the weighted hyperplane divisor of $\bP^2(1,1,2)$. Notice we indeed have $H'=\frac{1}{2}F=H|_{\bP^2(1,1,2)}$ under the embedding given above (taking into account of the $\bZ_2$-orbifold locus of $\bP^3(2,2,2,1)$). In any case we get $\beta=3/3=1$. So there is actually a K\"{a}hler-Einstein metric on the normal variety $\cX_0$ which is orbifold smooth along $(\cV_0)_\infty$. Now we can apply Theorem \ref{thmorb} to obtain the statement we wanted.

\begin{rem}
For all other cases of $(n,k)$, by the calculations in \cite{Li15a} it was conjectured that $\hvol$ is minimized at the valuation associated to the weight $\left(1, \cdots, 1, \frac{n-2}{n-1}\right)$. A more general question will be studied in a forth coming paper.
\end{rem}

\end{exmp}

\begin{exmp}
Next, we will consider quotient surface singularities. 
\begin{prop}\label{quotsurf}
Let $(X,o)$ be a quotient surface singularity with local analytic model $\bC^2/G$, where $G$ acts freely in codimension 1. Then
\[
\min_{v\in\Val_{X,o}}\hvol(v)= \frac{4}{|G|}.
\]
The minimum is achieved at the pushforward of $\ord_0\in\Val_{\bC^2,0}$.
\end{prop}

\begin{proof}
 For simplicity we may assume that $(X,o)=(\bC^2/G,0)$ and $G\subset U(2)$. Denote by $S^1\subset U(2)$ the subgroup consisting of diagonal matrices. Let $H:= G\cap S^1$ with $d:=|H|$. 
 Denote by $v_*$ the pushforward of $\ord_0\in\Val_{\bC^2,0}$.
 
 We first show that $\hvol(v_*)=\frac{4}{|G|}$.
 Let $\widehat{\bC^2}$ be the blow up of $\bC^2$ at the origin $0$ with exceptional divisor $E$. Denote by $\pi:\bC^2\to X$
 the quotient map. Then $\pi$ lifts to $\widehat{\bC^2}$ as $\hat{\pi}:\widehat{\bC^2}\to \widehat{X}$, where $\widehat{X}:=\widehat{\bC^2}/G$. We have the following commutative diagram:
 \[
 \begin{tikzcd}
  \widehat{\bC^2}\arrow{r}{\hat{\pi}}\arrow{d}[swap]{g} & 
  \widehat{X}
  \arrow{d}{h}\\
  \bC^2\arrow{r}{\pi} & X
 \end{tikzcd}
 \]
 Let $F\subset\widehat{X}$ be the exceptional divisor of $h$. For a general point on $F$, its stabilizer is exactly $H$. So $\hat{\pi}^*F=d E$, which implies that $v_*=\pi_*(\ord_E)=d~\ord_F$. It is clear that
 \[
  K_{\widehat{\bC^2}}=\hat{\pi}^*\left(K_{\widehat{X}}+\left(1-\frac{1}{d}\right) F \right).
 \]
 Then combining these equalities with $K_{\widehat{\bC^2}}=g^* K_{\bC^2}+E$, we get
 \[
  K_{\widehat{X}}=h^*K_{X} +\left(\frac{2}{d}-1\right) F.
 \]
 Hence $A_X(\ord_F)=\frac{2}{d}$ and $A_X(v_*)=d\cdot A_X(\ord_F)=2$. 
 Lemma \ref{quotmult} implies that $\vol(v_*)=\frac{1}{|G|}$.
 So $\hvol(v_*)=\frac{4}{|G|}$.
 
 Now we will show that $\hvol(v)\geq \frac{4}{|G|}$ for any real 
 valuation $v\in\Val_{X,o}$.
 Consider $\overline{X}:=\bP^2/G$ as a natural projective compactification of $X$. Let $\omega_{FS}$ be the Fubini-Study metric on $\bP^2$. Since $G$ acts isometrically on $(\bP^2,\omega_{FS})$, $\omega_{FS}$ induces an orbifold K\"ahler-Einstein metric $\omega_G$ on $\overline{X}$. Let $l_\infty$ be the line of infinity in $\bP^2$. The metric $\omega_G$ can be also viewed as a conical K\"ahler-Einstein metric on $\overline{X}$ with a cone angle $\frac{2\pi}{d}$ along $D:=l_\infty/G$. Therefore,
 $(\overline{X},(1-\frac{1}{d})D)$ is log Ding semistable. By Proposition \ref{proplb} we know that 
 \[
  \hvol(v)\geq \frac{4}{9} \left(\left(-K_{\overline{X}}-\left(1-\frac{1}{d}\right)D\right)^2\right).
 \]
 Let $\bar{\pi}:\bP^2\to \overline{X}$ be the quotient map. We notice that
 \[
  K_{\bP^2}=\bar{\pi}^*\left(K_{\overline{X}}+\left(1-\frac{1}{d}\right)D\right).
 \]
 Hence
 \[
  \hvol(v)\geq \frac{4}{9}\cdot\frac{((-K_{\bP^2})^2)}{|G|}=\frac{4}{|G|}.
 \]
 Thus we prove the proposition.
\end{proof}

\begin{lem}\label{quotmult}
Let $G$ be a finite group acting on $\bC^2=\Spec~\bC[x,y]$. 
Assume that the $G$-action is free in codimension 1. 
Then we have
\[
\lim_{m\to\infty}\frac{\dim_{\bC}\bC[x,y]^G_{< m}}{m^2/2}=\frac{1}{|G|}.
\]
\end{lem}

\begin{proof}
 Denote $W:=\bC[x,y]_1= \bC x\oplus\bC y$. Then we have that
\[ 
 \bC[x,y]\cong
 \bigoplus_{m\geq 0} \Sym^m W.
\]
Denote by $\rho_m:G\to GL(m+1,\bC)$ the representation of $G$ on $\Sym^m W$.
Since $G$ is a finite group, $\rho_m(g)$ is diagonalizable for any $m\geq 0$
and $g\in G$.
Denote the two eigenvalues of $\rho_1(g)$ by $\lambda_g$, $\mu_g$. Then
the eigenvalues of $\rho_m(g)$ are exactly $\lambda_g^m,
\lambda_g^{m-1}\mu_g,\cdots, \lambda_g\mu_g^{m-1},\mu_g^m$. From representation
theory we know that
\[
 \dim_{\bC} (\Sym^m W)^G=\frac{1}{|G|}\sum_{g\in G}\tr(\rho_m(g)).
\]
Let $d_m:=\dim_{\bC}\bC[x,y]^G_{< m}$. Hence we have
\begin{align*}
 d_m & =\sum_{i=0}^{m-1}\dim_{\bC} (\Sym^i W)^G=\frac{1}{|G|}\sum_{i=0}^{m-1}\sum_{g\in G}\tr(\rho_i(g))\\
  & =\frac{1}{|G|}\sum_{g\in G}\sum_{i=0}^{m-1}(\lambda_g^i+\lambda_g^{i-1}\mu_g
  +\cdots+\lambda_g\mu_g^{i-1}+\mu_g^i)\\
  & =\frac{1}{|G|}\sum_{g\in G}\sum_{i+j\leq m-1}\lambda_g^i\mu_g^j.
\end{align*}
The eigenvalues $\lambda_g,\mu_g$ can be chosen in a way so that 
$\lambda_{g^{-1}}=\lambda_g^{-1}$ and $\mu_{g^{-1}}=\mu_g^{-1}$. For 
any $m$ divisible by $|G|$, we have $\lambda_g^m=\mu_g^m=1$.
Hence we get:
\begin{align*}
 d_{m+1} & = \frac{1}{|G|}\sum_{g\in G}\sum_{i+j\leq m}\lambda_g^i\mu_g^j = \frac{1}{|G|}\sum_{g\in G}\sum_{i+j\leq m}\lambda_{g^{-1}}^i\mu_{g^{-1}}^j\\
 & = \frac{1}{|G|}\sum_{g\in G}\sum_{i+j\leq m}\lambda_g^{m-i}\mu_g^{m-j}.
\end{align*}
Therefore, for any $m$ divisible by $|G|$ we have
\begin{align*}
 d_m+d_{m+1} & =\frac{1}{|G|}\sum_{g\in G}\left(\sum_{i+j\leq m-1}\lambda_g^i\mu_g^j+
 \sum_{i+j\leq m}\lambda_g^{m-i}\mu_g^{m-j}\right)\\
 & = \frac{1}{|G|}\sum_{g\in G}\left(\sum_{i=0}^m\lambda_g^i\right)
 \left(\sum_{j=0}^m\mu_g^j\right).
\end{align*}
Since $G$ acts freely in codimension $1$, we have that $\lambda_g,
\mu_g\neq 1$ unless $g$ is the identity. Hence for any $g$ that is not
the identity, we have $\sum_{i=0}^m\lambda_g^i=1$ and $\sum_{j=0}^m\mu_g^j=1$.
As a result,
\begin{equation}\label{quoteq}
 d_m+d_{m+1}=\frac{1}{|G|}\left((m+1)^2+|G|-1\right).
\end{equation}
Since $\bC[x,y]^G$ is a finitely generated $\bC$-algebra, we know that
$d_m\sim cm^2$ for some constant $c$. Thus \eqref{quoteq} implies that
$c=\frac{1}{2|G|}$, which finishes the proof.
\end{proof}

\begin{rem}
 It is clear that $(\bC^2\setminus\{0\})/G\to \bP^1/G$ has a 
 natural Seifert $\bC^*$-bundle structure with $\ocC_\orb= \bP^2/G$
 and $V_\infty=l_\infty/G$. Hence Proposition \ref{quotsurf} can be also
 proved by applying Theorem \ref{thmorb}.
\end{rem}

\end{exmp}

\begin{exmp}[A logarithmic example]
Assume $V$ is a projective toric variety determined by a lattice polytope $P\subset \bR^{n-1}$ which is defined by the following linear functions on $\bR^{n-1}$:
\[
l_i(x):=\langle \eta_i, x\rangle+a_i\ge  0, \quad 1\le i\le N.
\]
where $\eta:=\{\eta_i\}$ is a set of primitive inward normal vectors to the facets of $P$. Here the primitivity of $\eta_i$ means that $a\cdot \eta_i\in \bZ^{n-1}$ for $a>0$ if and only if $a\in \bZ_{>0}$. 
We denote by $p_*$ the center of mass of $P$ with respect to the Lebesgue measure. 

By the works in \cite{Leg11, BB12, DGSW13}, there exists a conical K\"{a}hler-Einstein potential on the pair $(V, E; H)$ where $E=\sum_{i}(1-\gamma_i)E_i$ with $\gamma_i=r\cdot l_i(p_*)$ for any $r>0$ satisfying
$r\cdot l_i(p_*)\le 1$ for $1\le i\le N$. In this case, the log anti canonical class of $(V, E)$ is 
\[
-K_{(V,E)}=\sum_i E_i-\sum_i (1-\gamma_i)E_i=\sum_i \gamma_i E_i=r \cdot \sum_i L_i(p_*)E_i=r H,
\] 
which is ample. By Proposition \ref{cKEproj}, there exists a conical KE on the projective cone $(\ocC, \cE+(1-\beta)V_\infty)$ with $\cE=\sum_i (1- \gamma_i) \cE_i$ and $\beta=r/n$. We can also get this result by using the existence result in \cite{Leg11, BB12, DGSW13}. Indeed,
$\ocC$ is also a toric projective variety given the polytope $\cP$ in $\bR^n$ defined by the linear functions of $y=(y', y_n)\in \bR^{n-1}\times\bR$:
\[
L_i(y):=\langle \eta_i, y'\rangle+a_i y_{n}\ge 0; \quad  L_{N+1}:=-y_n+1\ge 0.
\]
It's elementary that the center of mass of $\cP$ is given by $\mathfrak{p}_*=\frac{n}{n+1}(p_*, 1)$. So by \cite{Leg11, BB12, DGSW13} there is a conical KE metric on the pair $(\ocC, D)$ where 
\[
D=\sum_i (1-\beta_i)\cE_i+(1-\beta_n)V_\infty.
\]
The cone angles are given by: 
\begin{eqnarray*}
\beta_i=s\cdot L_i(\mathfrak{p}_*)=s\left(\left\langle \eta_i, \frac{n}{n+1}p_*\right\rangle + a_i\frac{n}{n+1}\right)=s l_i(p_*)=\frac{sn}{r(n+1)}\gamma_i.
\end{eqnarray*} 
and 
\[
\beta_n=s\cdot L_n(\mathfrak{p}_*)=s\cdot \left(-\frac{n}{n+1}+1\right)=\frac{s}{n+1}.
\]
So if we choose $s=\frac{r(n+1)}{n}$ so that $\beta_i=\gamma_i$ for $1\le i\le n-1$, then $\beta_n=r/n$ as seen. 

By Theorem \ref{thmlog}, $\hvol_{((\cC, \cE)}$ obtains its global minimum at the valuation corresponding to the canonical $\bC^*$-action on the cone. 
\end{exmp}


\section{Derivative of volumes revisited}\label{secdervol}

In this section, we will point out the relation between the estimates in Section \ref{secvalfil} to the arguments/calculations in \cite{Li15b}. This will allow us to get a different proof of Theorem \ref{thmlog}. The main idea is similar to that in \cite{Li15a, Li15b}, that is the formula appearing in log version of Fujita's result should be viewed as a derivative of normalized volumes. Firstly, similar to \cite{Li15b} we will derive a formula for $\hvol_{(\cC, \cE)}(v)$ using the ``leading component" filtration. Then we consider a function $\Phi(s)$, which is convex with respect to $s\in [0,1]$  and interpolates: $\Phi(0)=\hvol_{(\cC, \cE)}(v_0)$ and $\Phi(1)=\hvol_{(\cC, \cE)}(v)$. To show $\Phi(1)\ge \Phi(0)$, by the convexity, we just need to show that the derivative of $\Phi$ at $s=0$ is nonnegative. Our main observation is that this non-negativity is exactly the inequality \eqref{Avint}. So from this point, we can use Proposition \ref{logFujval} and get a different proof of Theorem \ref{thmlog}.


\subsection{A volume formula}

Let $(V, H)$ be a polarized projective variety. Denote by $R=\bigoplus_{k=0}R_k=\bigoplus_{k=0}^{+\infty}H^0(V, kH)$ the graded section ring of $(V, H)$. 
Let $\mathcal{C}=C(V, H)={\rm Spec}(R)$ be the affine cone of $V$ with the polarization $H$.  Let $v_1$ be any real valuation centered at $o$ with $A_{\cC}(v_1)<+\infty$. We introduce the following notation. 
\begin{defn}
For any $g\in R=\bigoplus_{k=0}^{+\infty}R_k$, assume that it is decomposed into ``homogeneous"
components $g=g_{k_1}+\cdots+g_{k_p}$ with $g_{k_j}\neq 0\in R_{k_j}$ and $k_1<k_2<\cdots<k_p$. We define $\deg(g)=k_p$ and define $\bld(g)$ to be the ``leading component" $g_{k_p}$. 
\end{defn}
With the above definition, we define a filtration:
\begin{equation}\label{mirafil}
\cF^x R_k=\{f\in R_k; \exists g\in R \text{ such that } v_1(g)\ge x \text{ and } \bld(g)=f\}.
\end{equation}
We notice that following properties of $\cF^x R_{\bullet}$.
\begin{enumerate}
\item 
For fixed $k$, 
$\{\cF^x R_k\}_{x\in \bR}$ is a family of decreasing $\bC$-subspace of $R_k$. 
\item
$\cF$ is multiplicative:
$v_1(g_i)\ge x_i$ and $\bld(g_i)=f_i\in R_{k_i}$ implies 
\[
v_1(g_1 g_2)\ge x_1+x_2, \quad \bld(g_1 g_2)=f_1\cdot f_2\in R_{k_1+k_2}.
\]
\item 
If $A(v_1)<+\infty$, then $\cF$ is linearly bounded. Indeed by Izumi's theorem in \cite[Proposition 1.2]{Li15a} (see also \cite{Izu85, Ree89, BFJ12}), there exist $c_1, c_2\in (0, +\infty)$ such that
\[
c_1 v_0\le v_1\le c_2 v_0.
\]
For any $g\in R$, $k:=v_0(\bld(g))\ge v_0(g)\ge c_2^{-1}x$. So if $x>c_2 k$ then $\cF^x R_k=0$.
So $\cF$ is linearly bounded from above. On the other hand, for any $f\in R_k$, $v_1(f)\ge c_1 v_0(f)= c_1 k$. So if $x \le c_1 k$, then $\cF^x R_k=R_k$. So $\cF$ is linearly bounded from below. 
Note that the argument in particular shows the following relation:
\begin{equation}\label{cFlbd}
\inf_{\fm}\frac{v_1}{v_0}\le e_{\min}\le e_{\max}\le \sup_{\fm}\frac{v_1}{v_0}.
\end{equation}
\end{enumerate}
For later convenience, from now on we will fix the following constant:
\begin{equation}\label{defc1}
c_1:=\inf_{\fm}\frac{v_1}{v_0}>0.
\end{equation}
By \cite[Lemma 4.2]{Li15b}, we have the following characterization of this constant:
\begin{equation}\label{infv1}
c_1=v_1(V)=\inf_{i>0} \frac{v_1(R_i)}{i}.
\end{equation}
The following is a reason why the filtration in \eqref{mirafil} works for our purpose.
\begin{prop}[{cf. \cite[Proposition 4.3]{Li15b}}]\label{propdim}
For any $m\in \bR$, we have the following identity:
\[
\sum_{k=0}^{+\infty}\dim_{\bC}\left(R_k/\cF^m R_k\right)=\dim_{\bC}\left(R/\fa_m(v_1)\right).
\]
\end{prop}
Notice that because of linearly boundedness, the sum on the left hand side is a finite sum. More precisely, by the above discussion, when $k\ge m/c_1$, then $\dim_{\bC}\left(R_k/\cF^m R_k\right)=0$.
\begin{proof}
For each fixed $k$, let $d_k=\dim_{\bC}(R_k/\cF^m R_k)$. Then we can choose a basis of $R_k/\cF^m R_k$:
\[
\left\{[f^{(k)}_i]_k \; |\; f^{(k)}_i\in R_k, 1\le i\le d_k \right\},
\] 
where $[\cdot]_k$ means taking quotient class in $R_k/\cF^mR_k$. Notice that for $k\ge \lceil m/c_1 \rceil$, the set is empty.
We want to show that the set 
\[
\mathfrak{B}:=\left\{[f^{(k)}_i]\;|\; 1\le i\le d_k, 0\le k\le \lceil m/c_1\rceil-1 \right\}.
\] 
is a basis of $R/\fa_m(v_1)$, where $[\cdot]$ means taking quotient in $R/\fa_m(v_1)$. 
\begin{itemize}
\item
We first show that $\mathfrak{B}$ is a linearly independent set. 
For any nontrivial linear combination of $[f^{(k)}_i]$:
 \[
\sum_{k=0}^{N} \sum_{i=1}^{d_k}c^{(k)}_i [f^{(k)}_i]=\left[\sum_{k=0}^{N} \sum_{i=1}^{d_k} c^{(k)}_i f^{(k)}_i\right]=[f^{(k_1)}+\cdots+f^{(k_p)}]=:[F],
 \] 
where $f^{(k_j)}\neq 0\in R_{k_j}\setminus \cF^m R_{k_j}$ and 
$k_1<k_2<\cdots<k_p$. In particular $\bld(F)=f^{(k_p)}\not\in \cF^m R_{k_p}$. By the definition of $\cF^m R_{k_p}$, we know that
\[
f^{(k_1)}+\cdots+f^{(k_p)}\not\in \fa_m(v_1),
\]
which is equivalent to $[F]\neq 0\in R/\fa_m(v_1)$.
\item 
We still need to show that $\mathfrak{B}$ spans $R/\fa_m(v_1)$.
Suppose on the contrary $\mathfrak{B}$ does not span $R/\fa_m(v_1)$. Then there is some $k\in \bZ_{>0}$ and $f\in R_k-\fa_m(v_1)$ such that $[f]\neq 0\in R/\fa_m(v_1)$ can not
be written as a linear combination of $[f^{(k)}_i]$, i.e. not in the span of $\mathfrak{B}$. 
We can choose a minimal $k$ such that this happens. 
So from now on we assume that $k$ has been chosen such that that for any $k'<k$ and $g\in R_{k'}$, $[g]$ is in the span of $\mathfrak{B}$.
Then there are two cases to consider.
\begin{enumerate}
\item
If $f\in R_k\setminus \cF^m R_k$, then since $\{[f^{(k)}_i]_k\}$ is a basis of $R_k/\cF^m R_k$, we can write $f=\sum_{j=1}^{d_k} c_j f^{(k)}_j+h_k$ where $h_k\in \cF^m R_k$. By the definition of $\cF^m R_k$, there exists $h\in \fa_m(v_1)$ with $\bld(h)=h_k$. By the minimality of $k$, we know that $[h_k-h]=[h_k]$ is in the span of $\mathfrak{B}$. So $[f]=\sum_{j=1}^{d_k}c_j [f^{(k)}_j]+[h_k]$ is also in the span of $\mathfrak{B}$. 
Contradiction. 
\item
If $f\in \cF^m R_k\subset R_k$, then by the definition of $\cF^mR_k$, $f+h\in \fa_m(v_1)$ for some $h\in R$ such that $\bld(f+h)=f$. Since we assumed that
$[f]\neq 0\in R/\fa_m(v_1)$, we have $h\neq 0$ and $k':=\deg(h)<v_0(f)=k$. 
Now we can decompose $h$ into homogeneous components: 
\[
h=h^{(k_1)}+\cdots+h^{(k_p)},
\]
with $h^{(k_j)}\in R_{k_j}$ and $k_1<\cdots<k_p=k'$. Because $k'<k$ and the minimal property of $k$, we know that each $[h^{(k_j)}]$ in the span of $\mathfrak{B}$. 
So we have $[f]=[(f+h)-h]=[-h]$ is in the span of $\mathfrak{B}$. This contradicts our assumption that $[f]$ is not in the span of $\mathfrak{B}$. 
\end{enumerate}

\end{itemize}
\end{proof}
With the above proposition, we can follow \cite{Li15b} to derive the following volume formula:
\begin{eqnarray}\label{volv1a}
\vol(v_1)&=&\lim_{p\rightarrow +\infty} \frac{n!}{m^n}\dim_{\bC} R/\fa_{m}(v_1)\nonumber\\
&= &\frac{H^{n-1}}{c_1^n} - n\int_{c_1}^{+\infty}\vol\left(R^{(t)}\right)\frac{dt}{t^{n+1}}\\
&=&-\int_{c_1}^{+\infty} \frac{d\vol\left(R^{(t)}\right)}{t^n}.
\end{eqnarray}
As in \cite{Li15b}, motivated by the case of $\bC^*$-invariant valuations, we define a function of two parametric variables $(\lambda, s)\in (0,+\infty)\times [0,1]$:
\begin{eqnarray}\label{eqintfn}
\Phi(\lambda, s)&=&\frac{H^{n-1}}{(\lambda c_1 s+(1-s))^n}-n \int^{+\infty}_{c_1}
\vol\left(R^{(t)}\right)\frac{\lambda s dt}{(1-s+\lambda s t)^{n+1}}\nonumber\\
&=&\int^{+\infty}_{c_1} \frac{-d\vol\left(R^{(t)}\right)}{((1-s)+\lambda st)^{n}}.
\end{eqnarray}
Then it's easy to verify that $\Phi(\lambda, s)$ satisfies the following properties:
\begin{enumerate}
\item 
For any $\lambda\in (0,+\infty)$, we have:
\[
\Phi\left(\lambda, 1\right)= \vol(\lambda v_1)=\lambda^{-n}\vol(v_1), \quad \Phi(\lambda, 0)=\vol(v_0)=H^{n-1}.
\]
\item
For fixed $\lambda\in (0,+\infty)$, $\Phi(\lambda, s)$ is continuous and convex with respect to $s\in [0,1]$.
\item 
The directional derivative of $\Phi(\lambda, s)$ at $s=0$ is equal to:
\begin{equation}\label{dirder0}
\Phi_s(\lambda,0)=n\lambda H^{n-1}\left(\lambda^{-1}-c_1-\frac{1}{H^{n-1}}\int^{+\infty}_{c_1} \vol\left(R^{(t)}\right) dt \right).
\end{equation}
\end{enumerate}
Roughly speaking, the parameter $\lambda$ is a rescaling parameter, and $s$ is an interpolation parameter. To apply this lemma 
to our problem, we let $\lambda=\frac{r}{A_{(\cC, \cE)}(v_1)}=:\lambda_*$ such that
\[
\Phi(\lambda_*, 1)=\frac{A_{(\cC, \cE)}(v_1)^n\vol(v_1)}{r^n}=\frac{\hvol_{(\cC, \cE)}(v_1)}{r^n}.
\]
Recall that $\hvol_{(\cC,\cE)}(v_0)=r^n H^{n-1}=\hvol(v_0)$. So the problem of showing $\hvol_{(\cC, \cE)}(v_1)\ge \hvol_{(\cC, \cE)}(v_0)$ is equivalent to showing that $\Phi(\lambda_*, 1)\ge \Phi(\lambda_*, 0)$.
By item 2-3 of the above lemma, we just need to show that $\Phi_s(\lambda_*, 0)$ is non-negative. 
This will be finally proved in the next section by showing that the derivative $\Phi_s(\lambda_*, 1)$ is nothing but the 
left-hand-side of \eqref{Avint} for the case $(X, D)=(\ocC, (1-\beta)V_\infty+\cE)$.

Here to prepare for this calculation, we first transform the expression in \eqref{dirder0} into a different expression. For convenience, we first define a function:
\begin{equation}\label{eqTheta}
\Theta(t)=n\int^{+\infty}_{t}\vol(R^{(x)})\frac{t^n dx}{x^{n+1}}.
\end{equation}
Although this seems to be arbitrary at present, this turns out to be a natural function. Indeed, we will see that $\Theta(t)$ is nothing but the $\vol(\cF S^{(t)})$ (see \eqref{volFSt}).
Notice that by the volume formula \eqref{volv1a}, we have:
\begin{equation}\label{Thetac1}
\Theta(c_1)=c_1^n n \int^{+\infty}_{c_1}\vol(R^{(x)})\frac{dx}{x^{n+1}}=H^{n-1}-c_1^n \vol(v_1).
\end{equation}
Let's calculate the integral $\int^{+\infty}_{c_1}\Theta(t)dt$:
\begin{eqnarray*}
\int^{+\infty}_{c_1}\Theta(t)dt&=&n\int^{+\infty}_{c_1}dt\int^{+\infty}_{t}\vol\left(R^{(x)}\right)\frac{t^ndt}{x^{n+1}}\\
&=&n\int^{+\infty}_{c_1}\vol\left(R^{(x)}\right)\frac{dx}{x^{n+1}}\int^{x}_{c_1}t^ndt\\
&=&\frac{n}{n+1}\left(\int^{+\infty}_{c_1}\vol\left(R^{(x)}\right)dx-c_1\int^{+\infty}_{c_1}\vol\left(R^{(x)}\right)\frac{c_1^ndx}{x^{n+1}}\right)\\
&=&\frac{n}{n+1}\int^{+\infty}_{c_1}\vol\left(R^{(x)}\right)dx-\frac{c_1}{n+1}\Theta(c_1).
\end{eqnarray*}
So we get the formula:
\begin{equation}
\int^{+\infty}_{c_1}\vol\left(R^{(x)}\right)dx=\frac{n+1}{n}\int^{+\infty}_{c_1}\Theta(t)dt+\frac{c_1}{n}\Theta(c_1).
\end{equation}

Dividing the above formula by $H^{n-1}$ and substituting this into the \eqref{dirder0}, we get the new expression for the derivative $\Phi$ at $s=0$:
\begin{equation}\label{dirder0b1}
\Phi_s(\lambda,0)=n \lambda H^{n-1} \left(\lambda^{-1}-c_1-\frac{n+1}{n H^{n-1}}\int^{+\infty}_{c_1} \Theta dt-\frac{\Theta(c_1)}{nH^{n-1}}\right).
\end{equation} 
Using the formula \eqref{Thetac1}, we also get another expression:
\begin{equation}\label{dirder0b}
\Phi_s(\lambda,0)=n \lambda H^{n-1} \left(\lambda^{-1}-c_1\frac{n+1}{n}-\frac{n+1}{n H^{n-1}}\int^{+\infty}_{c_1} \Theta dt+\frac{c_1^{n+1}}{n}\frac{\vol(v_1)}{H^{n-1}}\right).
\end{equation} 


\subsection{Completion of the second proof of Theorem \ref{thmlog}}

Let $\ocC=\cC\cup V_\infty={\rm Proj}(S)$ be the projective cone as before, where 
\[
S=\bigoplus_{m=0}^{+\infty} S_m=\bigoplus_{k=0}^{+\infty} H^0(V, mH)
\] 
and $H=-r^{-1}(K_V+E)$ for $r\in \bQ_{>0}$. Recall that we have 
\[
K_{\ocC}+\cE+(1-\beta)V_\infty=-(r+1)V_\infty+(1-\frac{r}{n})V_\infty=-\frac{r(n+1)}{n}V_\infty.
\]
For simplicity of notations, we let $\delta=r(n+1)/n$ and $L=V_\infty$. Notice that $(L^n)=(H^{n-1})$.
\begin{prop}
We have the following equality:
\begin{equation}\label{dirder0c}
\Phi_s(\lambda,0)=n\lambda (H^{n-1})\left(\lambda^{-1}-\frac{n+1}{n(H^{n-1})}\int^{+\infty}_0\vol\left(\cF  S^{(t)}\right)dt\right),
\end{equation}
where the left-hand-side is the same as in \eqref{dirder0}.
\end{prop}
Assuming this formula, we can prove our Theorem \ref{thmlog}. Indeed, when $\lambda=\lambda_*=\frac{r}{A_{(\cC, \cE)}(v_1)}$, we have:
\begin{eqnarray*}
\Phi_s(\lambda_*,0)&=&\frac{nr (H^{n-1})}{A_{(\cC, \cE)}(v)}\left(\frac{A_{(\cC, \cE)}(v)}{r}-\frac{n+1}{n (H^{n-1})}\int^{+\infty}_0\vol\left(\cF S^{(t)}\right)dt\right)\\
&=&\frac{n(H^{n-1})}{A_{(\cC, \cE)}(v)}\left(A_{(\ocC, D)}(v)-\frac{\delta}{L^n}\int^{+\infty}_0 \vol\left(\cF S^{(t)}\right)dt\right),
\end{eqnarray*}
where $X=\ocC$ and $D=(1-\beta)V_\infty+\cE$. By Proposition \ref{logFujval}, the right-hand-side is nonnegative.


The rest of this section is devoted to the proof of the formula \eqref{dirder0c}. We start with the following observation, which can be seen as a cut-off version of Lemma \ref{propdim}.
\begin{lem}
For any $x\in \bR$, we have the following identity:
\begin{equation}\label{dimkFkR}
\dim_{\bC}\cF^{x} S_m=\sum_{k=0}^{m}\dim_{\bC} \cF^{x} R_k.
\end{equation}
\end{lem}
\begin{proof}
For any $k\in [0,m]\cap \bZ$, denote $l_k=\dim_{\bC}\cF^{x}R_k$ and let $\{f^{(k)}_i | 1\le i\le l_k\}$ be a basis of $\cF^{x}R_k$. By definition of $\cF^x R_k$, for each $g^{(k)}_i$ in the basis,
there exists $g^{(k)}_i\in R$ such that $v_1(g^{(k)}_i)\ge x$ and $\bld(g^{(k)}_i)=f^{(k)}_i$ (in particular $\deg(g^{(k)}_i)=\deg(f^{(k)}_i)\le m$). We claim that the set $\{g^{(k)}_i \;|\; 1\le i\le l_k, 1\le k\le m\}$ is a basis of 
$\cF^x S_m$. Indeed, it's easy to see that this set consists of linearly independent elements. We just need to show that they span $\cF^x  S_m$. For any $g\in \cF^x S_m$, we decompose $g$ 
into homogeneous components: $g=g_{k_1}+\cdots+g_{k_p}$ with $g_{k_j}\neq 0\in R_{k_j}$ and $k_1<\cdots<k_p=\deg(g)\le m$. So we have $\bld(g)\in \cF^x R_{k_p}$ is in the span of $f^{(k_p)}_i$.
So we can find a linear combination $h$ of $g^{(k_p)}_i$ with $\bld(h)=\bld(g)$. Now $\deg(g-h)<\deg(g)$ and $v_1(g-h)\ge \min\{v_1(g), v_1(h)\}\ge x$. In other words $g-h\in \cF^x  S_m$  is of strictly smaller degree compared to $g$. Now we can use induction to finish the proof.
\end{proof}
Using \eqref{dimkFkR} and Lemma 4.5 in \cite{Li15b}, we get the formula for $\vol(\cF S^{(x)})$:
\begin{eqnarray}\label{volFSt}
\vol\left(\cF S^{(x)}\right)= n \int^{+\infty}_{x} \vol\left(R^{(t)}\right)\frac{x^n dt}{t^{n+1}}=\Theta(x),
\end{eqnarray}
where we used the previously defined function $\Theta(t)$ in \eqref{eqTheta}.
\begin{cor}
The following holds for $\vol(\cF S^{(x)})$:
\begin{enumerate}
\item If $x\ge c_2$, then 
$\vol(\cF S^{(x)})=0.$
\item 
For any $x\in \bR$, we have:
\[
\vol(\cF S^{(x)})\ge -x^n\vol(v_1)+H^{n-1}
\]
The equality holds if $x\le c_1$.
\end{enumerate}
\end{cor}
\begin{proof}
If $x\le c_1$, then with the help of the volume formula \eqref{volv1a}, we get:
\begin{eqnarray*}
\vol(\cF S^{(x)})&=&n x^n \int^{+\infty}_{c_1}\vol(R^{(t)})\frac{dt}{t^{n+1}}+n x^n H^{n-1}\int^{c_1}_{x}\frac{dt}{t^{n+1}}\\
&=&x^n \left(\frac{H^{n-1}}{c_1^n}-\vol(v_1)\right)+H^{n-1} x^n \left[-t^{-n}\right]^{c_1}_{x}\\
&=&x^n\left(\frac{H^{n-1}}{c_1^n}-\vol(v_1)\right)+H^{n-1}\left(-\frac{x^n}{c_1^n}+1\right)\\
&=&-x^n \vol(v_1)+ H^{n-1}.
\end{eqnarray*}
If $x\ge c_1$ then because $\vol(R^{(t)})\le H^{n-1}$, we have:
\begin{eqnarray*}
\vol\left(\cF S^{(x)}\right)&=&n\int^{+\infty}_{x}\vol\left(R^{(t)}\right)\frac{x^n dt}{t^{n+1}}=n x^n\left(\int^{+\infty}_{c_1}- \int^{x}_{c_1}\right)\\
&\ge &x^n \left(\frac{H^{n-1}}{c_1^n}-\vol(v_1)\right)- x^n H^{n-1} \left[-t^{-n}\right]^{x}_{c_1}\\
&=& x^n \left(\frac{H^{n-1}}{c_1^n}-\vol(v_1)\right)+H^{n-1}\left(1-\frac{x^n}{c_1^n}\right)\\
&=& H^{n-1}-\vol(v_1)x^n.
\end{eqnarray*}
\end{proof}

Using the above lemma, we can derive a key ingredient in our formula:
\begin{eqnarray*}
\int_0^{+\infty}\vol(\cF S^{(x)})dx&=&\left(\int_0^{c_1}+\int_{c_1}^{+\infty}\right)\vol(\cF S^{(x)})dx\\
&=&-\vol(v_1)\int_0^{c_1}x^n dx+ H^{n-1} c_1+\int_{c_1}^{+\infty}\vol(\cF S^{(x)})dx\\
&=&-\frac{\vol(v_1)}{n+1}c_1^{n+1}+c_1  H^{n-1} +\int^{+\infty}_{c_1}\Theta(x)dx\\
&=&\left(-\frac{c_1^{n+1}}{n+1}\vol(v_1)+c_1 H^{n-1}+\int^{+\infty}_{c_1}\Theta(t)dt\right).
\end{eqnarray*}
The proof of \eqref{dirder0c} is finally at our hand, because we can transform the expression in the bracket of \eqref{dirder0b} to become:
\begin{eqnarray*}
&&\lambda^{-1}-c_1\frac{n+1}{n}-\frac{n+1}{n H^{n-1}}\int^{+\infty}_{c_1} \Theta dt+\frac{c_1^{n+1}}{n}\frac{\vol(v_1)}{H^{n-1}}\\
&=&\lambda^{-1} -\frac{n+1}{n H^{n-1}}\left(c_1 H^{n-1}+\int^{+\infty}_{c_1}\Theta dt-\frac{c_1^{n+1}}{n+1} \vol(v_1)\right)\\
&=&\lambda^{-1}-\frac{n+1}{n H^{n-1}}\int^{+\infty}_0\vol(\cF S^{(x)})dx\\
&=&\lambda^{-1}-\frac{n+1}{n}\frac{1}{H^{n-1}}\int^{+\infty}_0\vol(\cF S^{(x)})dx.
\end{eqnarray*}

\section{Minimizers from smooth Sasaki-Einstein metrics}

In this section, we generalize our minimization to the set-up of smooth Sasaki-Einstein metrics, or equivalently to the case of Ricci-flat K\"{a}hler cone metrics with isolated singularities. 

Let $X$ be an $n$-dimensional affine variety with an isolated singularity at $o$. Assume the algebraic torus $T_{\bC}:=(\bC^*)^r$ acts on $X$ such that $o$ is in the closure of any $T_\bC$-orbit. Assume that there is a $T_\bC$-equivariant holomorphic $(n,0)$-form $\sigma$ on $X$. The latter assumption allows us to define the weight function:
\[
A: \mathfrak{t}_{\bC} \rightarrow \bC, \quad u\mapsto \frac{\cL_u\sigma}{\sigma},
\]
where $\cL_u$ is the Lie derivative with respect to the holomorphic vector field $u$ in the Lie algebra $\ft_\bC$.

Suppose that there is a K\"{a}hler cone metric on $X$ with the radius function $r: X\rightarrow \bR_{\ge 0}$ and denote the link $M=\{r=1\}$. Then $M$ is a Sasaki manifold with the contact form $\eta=d^c\log r=-r^{-1} J dr$. The Reeb vector field, defined as $\xi=J(r\partial_r)$, is a holomorphic Killing vector field and satisfies $\eta(\xi)|_L=-J dr(J\partial_r)=1$. We assume that the holomorphic vector field $u=r \partial_r-i J(r\partial_r)$ is in the Lie algebra of $T_\bC$. The volume of the contact manifold $(M, \eta)$ is then equal to:
\[
\vol(\eta)=\vol(M, \eta)=\int_M \eta\wedge (d\eta)^{n-1}=\int_M d^c r\wedge (dd^cr)^{n-1}.
\]
If $r$ is a radius function of a K\"{a}hler cone metric, then for any $\lambda>0$, $\tilde{r}=r^\lambda$ is also a K\"{a}hler cone radius function. If we denote $\tilde{M}=\tilde{r}^{-1}(1)$ and $\eta=d^c\log \tilde{r}$, then we have the re-scaling properties $\vol(\tilde{M})=\lambda^{n} \vol(M)$ and $\tilde{u}=\tilde{r}\partial_{\tilde{r}}-i J(\tilde{r}\partial_{\tilde{r}})=\lambda^{-1} u$.

On the other hand, if there are two K\"{a}hler cone metric $\omega$ and $\tilde{\omega}$ with the same Reeb vector vector field 
$J(r\partial_r)=J(\tilde{r}\partial_{\tilde{r}})$, then $\tilde{r}=r e^{\phi}$ for a basic function $\phi$. In other words, $\phi$ satisfies $\mathcal{L}_\xi \phi=\cL_{\partial_r}\phi=0$). One can then verify that 
$\vol(M, \eta)=\vol(M, \tilde{\eta})$. So the volume function descends to become a function defined on the space of Reeb vector fields (see \cite{MSY08}). 

As in \cite{MSY08}, we define the space of Reeb vector fields as the dual cone of the moment cone of the $(S^1)^d$ action on $X$. 
It's well known that we have the following equivalent characterization of the Reeb cone. Suppose $X={\rm Spec}(R)$ for a finite generated $\bC$-algebra. Under the
torus action, we have a decomposition of $R$ into weight spaces:
\[
R=\bigoplus_{\alpha\in \Gamma} R_\alpha
\]
where
\begin{itemize}
\item For any $\alpha\in \bZ^d$, $R_\alpha=\{f\in R; t\circ f=t^\alpha f \text{ for any } t\in T_\bC=(\bC^*)^r\}$;
\item $\Gamma=\{\alpha\in \bZ^d; R_\alpha\neq 0\}$.
\end{itemize}
The Reeb cone of $X$ is then the following conic subset of the real Lie algebra:
\[
\ft^+_{\bR}=\{\xi\in \bR^d; \langle \alpha, \xi\rangle>0 \text{ for any } \alpha\in \Gamma\}.
\]
As a functional defined on $\ft^+_\bR$, the volume functional satisfies the rescaling property $\vol(\lambda \xi)=\lambda^{-n}\vol(\xi)$. Now consider the set of normalize Reeb vector fields:
\[
\hat{\ft}^+_\bR=\{\xi\in \ft; \cL_{\xi}\Omega=n \Omega\}.
\]
For any $\xi\in \ft^+_\bR$, denote by $\hat{\xi}:=\frac{n}{A(\xi)}\xi$ the corresponding normalized element in $\hat{\ft}^+_{\bR}$.
The following basic result in Sasaki-Einstein geometry was proved by Martelli-Sparks-Yau:
\begin{thm}[\cite{MSY08}]\label{thm-MSY}
If there is a Ricci-flat K\"{a}hler cone metric on $X$ with the Reeb vector field $\xi\in \ft$, then $\hat{\xi}$ minimizes the volume functional restricted to $\hat{\ft}^+_\bR$.
\end{thm}

Now we switch our points of view by using the normalized volume. As in \cite{Li15a}, the first observation is that any $\xi\in \mathfrak{t}^+_\bR$ determines a valuation $v_\xi$ of the function field $\bC(X)$ as follows:
\[
v_\xi(f)=\min\left\{\langle \alpha, \xi\rangle; f=\sum_\alpha f_\alpha \text{ with } f_\alpha\neq 0\right\}.
\]
$\xi\in \ft^+_\bR$ exactly means that the center of $v_\xi$ over $X$ is $o$. Moreover $A(\xi)=A(v_\xi)$ is nothing but the log discrepancy of the valuation $v_\xi$.
The main result of this section is the following
\begin{thm}\label{thm-irSE}
Notations as above, 
if there exists a Ricci-flat K\"{a}hler cone metric on $X$ with the Reeb vector field $\xi\in \ft^+_{\bR}$, then $v_{\xi}$ minimizes $\hvol(v)$ over $\Val_{X, o}$. 
\end{thm}

This is a strengthing of Theorem \ref{thm-MSY}. The rest of this section is devoted to the proof of Theorem \ref{thm-irSE}.   
The idea of proof for the irregular case is to use approximation by quasi-regular ones, similar to the one used in \cite{CS12} (see also \cite{CS15}).

If $\xi=J(r\partial_r)$ is a quasi-regular Reeb vector field, then the holomorphic field $u=r\partial_r-i \xi$ generates a $(\bC^*)^r$-action on $X$ such that $X/\bC^*=(S, \Delta)$ is a Fano orbifold and the projection $\pi=\pi_S: X\rightarrow S$ is an orbifold line bundle, which we will denote by $H$, or $H_S$ if we want to emphasize its dependence on $S$. Notice that if we rescale $\xi$ by a constant $\lambda>0$, the projection and the orbifold line bundle do not change. 
Moreover there is a canonical holomorphic vector field along the fibre of $H_S$ given by $u_*=\zeta\partial_\zeta$ where $\zeta$ is the coordinate variable along the fibre. In general, if $\xi=\lambda^{-1}\xi_*$, we have $u=\lambda^{-1} u_*$. One can then verify that $r^{2/\lambda}=h$ defines an orbifold Hermitian metric on $H_S^{-1}\rightarrow S$. So the contact form
$\eta=d^c\log r=\frac{\lambda}{2}d^c \log h$ satisfies
\[
d\eta=\frac{\lambda}{2}dd^c\log h=\lambda \sddb \log h = \pi^*(\lambda \omega_h).
\]
Hence the transverse K\"{a}hler metric $\omega^T$ can be identified with the orbifold K\"{a}hler metric $\lambda \pi_S^*\omega_h$. The K\"{a}hler cone metric on $X$ is given by:
\begin{equation}\label{eq-Omega}
\Omega=\sddb r^2=\sddb h^\lambda= \lambda  h^\lambda \pi_S^*\omega_h + \lambda^2 h^\lambda \frac{\sqrt{-1} \nabla \zeta\wedge \overline{\nabla \zeta}}{|\zeta|^2},
\end{equation}
where $\nabla \zeta=d\zeta+\zeta\cdot \partial \log h$. 

\begin{lem}\label{lem-Tclass}
If $\xi$ is quasi-regular with $A(\xi)=n$, then we have:
\begin{enumerate}
\item
 $\omega_M^T\in 2 \pi c_1(-(K_S+\Delta))/n$;
 \item $
\vol(\xi)=(2\pi)^n A(S)^n (-S|_S)^{n-1}/n^n=\frac{(2\pi)^n}{n^n} \hvol(\ord_S)$. 
\end{enumerate}
\end{lem}
\begin{proof}
The canonical vector field $\xi_*=\zeta\partial_\zeta$ has weight equal to $A(S)$ and $[\omega_h]=[d\eta_*|_\cH]=2\pi c_1(-S|_S)$. So if $A(\xi)=n$, then $\xi=\frac{n}{A(S)}\xi_*$ and $\eta=\frac{A(S)}{n}\eta_*$.
So we have:
\[
[\omega^T_M]=[d\eta|_\cH]=A(S) [d\eta_*|_\cH]/n\in A(S) 2\pi c_1(-S|_S)=2\pi c_1(-(K_S+\Delta))/n.
\]

To calculate $\vol(\xi)$ or equivalently $\vol(\eta)$, we notice that the K\"{a}hler cone metric tensor is given by:
\[
g_\Omega=dr^2+r^2\left(\lambda g_{\omega_h}+\lambda^2 \nabla\theta\otimes \nabla\theta \right).
\]
with $\lambda=\frac{A(S)}{n}$ and $\nabla\theta$ being the connection form of the orbifold $S^1$-bundle.
So the volume form $d{\rm  vol}_{g_M}=\eta\wedge (d\eta)^{n-1}$ on the link $M=r^{-1}(1)$ is given by 
$\lambda^{n} d{\rm vol}g_{\omega_h}\wedge d\theta$, whose integral is equal to 
$\lambda^n [\omega_h]^{n-1} 2\pi=\frac{A(S)^n}{n^n}(2\pi)^n (-S|_S)^{n-1}$.

\end{proof}

For any Fano orbifold (e.g. $X/\langle e^{t\xi}\rangle=(S, \Delta)$), we define the greatest lower bound of the Ricci curvature (see \cite{Tia92, Sze11}:
\[
R(S, \Delta)=\sup\{t>0; \exists \text{ an orbifold K\"{a}hler metric } \omega \text{ with } [\omega]=2\pi c_1(-K_{(S,\Delta)}) \text{ and } Ric(\omega)\ge t \omega \}.
\]
Taking $m\gg 1$ sufficiently divisible, we can choose an orbifold smooth divisor $D\in\left|-m K^{\rm orb}_{(S, \Delta)}\right|$, and define: 
\[
 R((S, \Delta), D/m)=\sup\left\{ \gamma >0; \exists \text{ an orbifold conical K\"{a}hler metric on } (S, \Delta+\frac{1-\gamma}{m}D)  \right\}.
\]

We have the orbifold version of an result in \cite{SW12}.
\begin{lem}[cf. \cite{SW12, Li13}]\label{lem-Rapp}
For any orbifold smooth divisor $D\in \left|-m K^{\orb}_{(S,\Delta)}\right|$, the following inequality holds:
\begin{equation}\label{eq-2Rorb}
R(S, \Delta)\ge R((S, \Delta); D/m)\ge \frac{(m-1)R(S, \Delta)}{m-R(S, \Delta)}.
\end{equation}
\end{lem}

\begin{proof}

Consider the log-Ding-energy:
\[
G_{\gamma} (\vphi):=F^0_{\phi_0}(\vphi)-\frac{1}{\gamma}\log\left(\int_S \frac{e^{-\gamma \vphi}}{|s|^{2(1-\gamma)/m}} \right)
\]
\[
F_{t}(\vphi):=F^0_{\phi_0}(\vphi)-\frac{1}{t}\log\left(\int_S e^{-t \vphi-(1-t)\phi_0} \right),
\]
where $D=\{s=0\}$, $e^{-\phi_0}$ is an orbifold smooth Hermitian metric on $-(K_S+\Delta)$ and $e^{-\vphi}$ is a continuous Hermitian metric on the line bundle $-(K_S+\Delta)$.

The first inequality follows from the inequality:
\[
G_t(\vphi)\le F_t(\vphi)+C_t,
\]
where $C_t$ is independent of $\vphi$. By H\"{o}lder's inequality, we have:
\begin{eqnarray*}
\int_S \frac{e^{-\gamma\vphi}}{|s|^{2(1-\gamma)/m}}&=& \int_S e^{-\gamma (\vphi-\phi_0)} \frac{1} {(|s|^2 e^{-m \phi_0})^{(1-\gamma)/m}}{} e^{-\phi_0}\\
&\le& \left(\int_S e^{-\gamma p (\vphi-\phi_0)}e^{-\phi_0}\right)^{1/p} 
\left(\int_S \frac{1}{(|s|^2 e^{-m\phi_0})^{(1-\gamma)q/m}} e^{-\phi_0}\right)^{1/q}.
\end{eqnarray*}

The second inequality follows by solving the inequalities:
\[
\gamma p< R(S, \Delta), \quad (1-\gamma)q/m<1, \quad p^{-1}+q^{-1}=1.
\]
\end{proof}

\begin{prop}\label{prop-hvolvR}
With the above notations, we have the following inequality
$$\hvol_X(v)\ge R(S, \Delta)^n A(S)^n (-S|_S)^{n-1}=R(S, \Delta)^n \hvol(\ord_S).$$
\end{prop}
\begin{proof}
For any $\gamma<R(S, \Delta)$, when $m$ is sufficiently large, there exists a conical K\"{a}hler-Einstein metric $\omega\in 2\pi c_1(-(K_X+\Delta))$ on 
$(S, \Delta+(1-\gamma)\frac{D}{m})$ with Ricci curvature $\gamma$:
\[
Ric(\omega)=\gamma \omega+\{\Delta\}+\frac{1-\gamma}{m}\{D\}.
\]
Notice that the following identities hold:
\[
-(K_S+\Delta)=-A(S) S|_S, \quad \hvol(\ord_S)=A(S)^n (-S|_S)^{n-1}.
\]
So we have
\[
-(K_S+\Delta+\frac{1-\gamma}{m}D)=-\gamma (K_S+\Delta)=-\gamma A(S)S.
\]
Denoting by $\mathcal{D}$ the corresponding divisor associated to $D$ on $X$. So by Theorem \ref{thmorb} we have:
\[
\hvol_{(X, \frac{1-\gamma}{m}\cD)}(v)\ge \gamma^n A(S)^n (-S|_S)^{n-1}=\gamma^n  \hvol(\ord_S).
\]

This holds for any $\gamma< R((S, \Delta); D/m)$. Moreover, we can choose $D$ sufficiently general such that $D$ is not contained in the center of $v$ so that $\hvol_{(X, \frac{1-\gamma}{m}\cD)}(v)=\hvol_{X}(v)$.  By letting $\gamma\rightarrow R((S, \Delta); D/m)$, we get
\[
\hvol_X(v)\ge R((S, \Delta); D/m)^n A(S)^n (-S|_S)^{n-1}.
\]
By Lemma \ref{lem-Rapp}, by choosing $m$ sufficiently large (and $D$ sufficiently general), we get:
\[
\hvol_X(v)\ge R(S, \Delta)^n A(S)^n (-S|_S)^{n-1}=R(S, \Delta)^n \hvol(\ord_S).
\]

\end{proof}

Finally we will deal with the irregular case. Recall that the following well known lemma.
\begin{lem}[see \cite{BGM06}]
With the above notations, the following conditions are equivalent:
\begin{enumerate}
\item
$(X, \Omega)$ is a Ricci-flat K\"{a}hler cone metric;
\item
 $(M, g_M)$ is an Sasaki-Einstein metric with Einstein constant equal to $2n-2$;
\item The transverse metric satisfies the identity $Ric(g^T_M)=2n g^T_M$.
\end{enumerate}  
Moreover if the Reeb vector field $\xi$ is quasi-regular, then the above condition is also equivalent to the condition that $((S, \Delta), \omega_h)$ is a K\"{a}hler-Einstein Fano orbifold satisfying the identity $Ric(\omega_h)=n\lambda \omega_h$.
\end{lem}

\begin{proof}
Equivalence of 1 and 2 follows from the Gauss-Codazzi equation for Ricci curvature. 
The equivalence of 2 and 3 follows from the formula for the Ricci curvatures:
\begin{enumerate}
\item[(a)]
$Ric(g_M)(X, \xi)=2n \eta(X)$ for any vector field $X$;
\item[(b)]
$Ric(g_M)(X, Y)=Ric(g^T_M)(X, Y)- 2g_M(X, Y)$ for any pair of sections $X, Y$ of $\cH$.
\end{enumerate}
In the quasi-regular case, $\omega_h=\omega^T_M$ and hence the last statement.
\end{proof}
\begin{rem}
In the quasi-regular case, equivalence of 1 and 3 also follows from the following calculation:
\begin{eqnarray*}
Ric(\Omega)&=&-\sddb\log \Omega^n=-\sddb\log (h^{n\lambda} \omega_h^n)\\
&=&\pi^*(Ric(\omega_h)-n\lambda \omega_h).
\end{eqnarray*}
\end{rem}

\begin{lem}\label{lem-MetricApp}
If $(M, g_M)$ is a Sasaki metric with an irregular Reeb vector field $\xi$, then there exists a sequence of quasi-regular Sasaki metrics $\{g_k\}_k\in \mathbb{N}$ on $M$ with Reeb vector fields $\xi_k$ such that $\xi_k\rightarrow \xi$ and $g_{M, k}\rightarrow g_M$ in the $C^\infty$ topology as $k\rightarrow+\infty$. Moreover we can assume $A(\xi_k)=A(\xi)$ for any $k$.
\end{lem}
\begin{proof}
We consider the deformations of Sasaki structures that preserve the CR structure and change the Reeb vector field (see \cite{BGM06, BGS08}):
\[
\tilde{\eta}=f\eta, \quad \tilde{\xi}=\xi+\rho, \quad \tilde{\Phi}=\Phi-\Phi \tilde{\xi}\otimes\tilde{\eta}
\]
such that:
\[
\tilde{\eta}(\tilde{\xi})=1, \quad \tilde{\xi}\rfloor d\tilde{\eta}=0, \quad \mathcal{L}_{\tilde{\xi}}\tilde{\Phi}=0.
\]
Then $f=\frac{1}{1+\eta(\rho)}$.  The corresponding Riemannian metric $\tilde{g}_M$ is given by
\begin{eqnarray*}
\tilde{g}_M=d\tilde{\eta}\circ (\tilde{\Phi}\otimes \mathbb{I})\oplus \tilde{\eta}\otimes\tilde{\eta}.
\end{eqnarray*}
It's easy to see that $\tilde{g}_M$ depends on $\tilde{\xi}$ smoothly.
When $\rho=\tilde{\xi}-\xi$ is sufficiently small, $\tilde{g}_M-g_M$ is sufficiently small in the $C^\infty$ topology. We can also let $\tilde{\xi}$ changes in the set $\widehat{\mathcal{C}}$ to ensure
$A(\xi_k)=A(\xi)=n$.

\end{proof}

Now we can complete the proof of Theorem \ref{thm-irSE}. 
\begin{proof}[Proof of the Theorem \ref{thm-irSE}]
As mentioned above, the transverse Ricci curvature and the Ricci curvature are related by:
\begin{equation}\label{eq-2Ric}
Ric(g_M)|_{\cH\times\cH}=Ric(g^T_M)-2 g^T_M,
\end{equation}
where $\cH$ denotes the horizontal distribution of the Sasakian structure.
Let $g_M$ be the Sasaki-Einstein metric associated to the Ricci-flat K\"{a}hler cone metric. By Lemma \ref{lem-MetricApp} we can find a sequence of quasi-regular Sasaki metrics $g_{M,k}$ with the corresponding quotients $((S_k, \Delta_k); \omega_k)$ satisfying:
$Ric(\omega_k)\ge (1-\epsilon)\omega_k $ where $\omega_k\in 2\pi c_1(-(K_{S_k}+\Delta_k))$. Indeed, we can choose $k$ large enough such that $g_{M, k}$ from Lemma \ref{lem-MetricApp} satisfies:
$Ric(g_{M, k})\ge (2(n-1)-2 n \epsilon)g_{M, k}$ which implies by \eqref{eq-2Ric} $Ric(g^T_{M, k})\ge (2n-2n \epsilon)g^T_{M, k}$ or equivalently 
$Ric(\omega^T_k)\ge (n-n \epsilon) \omega^T_{M,k}$ where $\omega^T_{M,k}$ is the transversal K\"{a}hler form associated to the transverse K\"{a}hler metric $g^T_{M,k}$. Moreover, we can assume $A(\xi_k)=n=A(\xi)$. Then by Lemma \ref{lem-Tclass}, we have $[\omega^T_k]\in 2\pi c_1(-(K_S+\Delta))/n$. We just need to let
$\omega_k=n \omega^T_{M,k}$.

By construction, there exists $c_k>0$ such that $v_{\xi_k}=c_k\cdot \ord_{S_k}$. Then as $k\rightarrow+\infty$, we have $c_k\cdot \ord_{S_k}\rightarrow v_\xi$ , $R(S_k, \Delta_k)\rightarrow 1$ and $\hvol_X(\ord_{S_k})=\hvol_X(v_{\xi_k}) \rightarrow \hvol_X(v_\xi)$. By Proposition \ref{prop-hvolvR}, we have
\begin{eqnarray*}
\hvol_X(v)&\ge&  R(S_k, \Delta_k)^n\hvol(S_k).
\end{eqnarray*}
Letting $k\rightarrow+\infty$, we get the wanted inequality: $\hvol_X(v)\ge \hvol(v_\xi)$.
\end{proof}




\section{A question}
We end this paper with the following question whose answer would lead to a purely algebraic proof of the result in this paper.

\noindent
{\bf Question}: With the notations used in this paper, give an algebraic proof (i.e. without using conical K\"{a}hler-Einstein metrics) of the following result (and its logarithmic/orbifold version): a $\bQ$-Fano variety $V$ is K-semistable (resp. K-polystable) if and only if $(\ocC, (1-\beta)V_\infty)$ is log-K-semistable (resp. log-K-polystable). 

Notice that the correct cone angle $\beta=\frac{r}{n}$ can be detected by log Futaki invariant defined in \cite{Don12}  as in \cite[Section 3.3]{LS12} (see also \cite{Li11}).

{\bf Postscript note:} 
After the completion of the first version of this paper, some consequences/applications of our main results has appeared in \cite{Liu16, LX16, HS16}. Moreover the above question has been answered affirmatively in \cite{LX16}.

\section{Acknowledgment}
We would like to thank Chenyang Xu and Kento Fujita for helpful comments.
The first author is partially supported by NSF DMS-1405936. 
The first author would like to thank Laszlo Lempert and Sai-Kee 
Yeung for their interest in this work. Part of this paper is written while the 
first author visits MSRI at Berkeley, and he would like to thank the institute for
its hospitality. The second author is partially
supported by NSF DMS-0968337. The second author would like to thank
his advisor J\'anos Koll\'ar for his constant support, encouragement and 
numerous inspiring conversations. The second author also wishes to thank
Charles Stibitz, Yury Ustinovskiy, Xiaowei Wang and Ziquan Zhuang for many useful 
discussions.

\noindent
Department of Mathematics, Purdue University, West Lafayette, IN 47907-2067

\noindent
{\it E-mail address:} li2285@purdue.edu

\vskip 2mm

\noindent
Department of Mathematics, Princeton University, Princeton, NJ, 08544-1000.

\noindent {\it E-mail address:} yuchenl@math.princeton.edu

\end{document}